\documentclass{amsart}
\usepackage{enumitem}
\usepackage{graphicx}	
\usepackage{amsmath, amsthm, amssymb,mathtools}
\usepackage{amsfonts}
\usepackage{hyperref}
\usepackage{algpseudocode}
\usepackage{mathrsfs}
\usepackage[capitalise,noabbrev]{cleveref}
\usepackage[linesnumbered,ruled,vlined]{algorithm2e} 

\newcommand{\algname}[1]{\texttt{#1}}
\numberwithin{equation}{section}

\newcommand{\rpart}{\boldsymbol{\lambda}}

\newcommand{\E}{\mathbb{E}}
\newcommand{\B}{\mathscr{B}}

\theoremstyle{plain}
\newtheorem{theorem}{Theorem}[section]
\newtheorem{proposition}[theorem]{Proposition}
\newtheorem{corollary}[theorem]{Corollary}

\newtheorem{lemma}[theorem]{Lemma}

\theoremstyle{definition}
\newtheorem{definition}[theorem]{Definition}
\newtheorem{remark}[theorem]{Remark}

\newtheorem{problem}[theorem]{Problem}

\newtheorem{example}[theorem]{Example}

\DeclareMathOperator{\Des}{Des}
\DeclareMathOperator{\des}{des}

\DeclareMathOperator{\maj}{maj}
\DeclareMathOperator{\col}{col}
\DeclareMathOperator{\fmaj}{fmaj}
\DeclareMathOperator{\pr}{Pr}
\DeclareMathOperator{\Ess}{Ess}
\DeclareMathOperator{\stat}{stat}

\title[Descents and flag major index on conjugacy classes of $\mathfrak{S}_{n,r}$]{Descents and flag major index on conjugacy classes of colored permutation groups \\ without short cycles}

\author{Kevin Liu}
\address{Department of Mathematics and Computer Science, The University Of The South}
\email{keliu@sewanee.edu}
\author{Mei Yin} 
\address{Department of Mathematics, University of Denver}
\email{mei.yin@du.edu}

\begin{document}

\begin{abstract}
    We consider the descent and flag major index statistics on the colored permutation groups, which are wreath products of the form $\mathfrak{S}_{n,r}=\mathbb{Z}_r\wr \mathfrak{S}_n$. We show that the $k$-th moments of these statistics on $\mathfrak{S}_{n,r}$ will coincide with the corresponding moments on all conjugacy classes without cycles of lengths $1,2,\ldots,2k$. Using this, we establish the asymptotic normality of the descent and flag major index statistics on conjugacy classes of $\mathfrak{S}_{n,r}$ with sufficiently long cycles. Our results generalize prior work of Fulman involving the descent and major index statistics on the symmetric group $\mathfrak{S}_n$. Our methods involve an intricate extension of Fulman's work on $\mathfrak{S}_n$ combined with the theory of the degree for a colored permutation statistic, as introduced by Campion Loth, Levet, Liu, Sundaram, and Yin.
\end{abstract}


\maketitle 

\section{Introduction}

Statistics on the symmetric group $\mathfrak{S}_n$ are a major area of study in combinatorics, and three commonly studied statistics are the descent, inversion, and major index statistics. Descents appear in the study of card shuffling \cite{BayerDiaconis}, carries when adding numbers \cite{carries_shuffling,holte}, and flag varieties \cite{FULMAN1999390}. Inversions appear in the study of sorting objects \cite{KnuthTAOCP3} and testing randomness \cite{inversions}. The major index statistic was originally introduced by MacMahon \cite{majorindex}, who showed that it is equidistributed with the inversion statistic on $\mathfrak{S}_n$. Since then, the major index and its variations appeared in the study of random tableau \cite{strahov}, flag manifolds \cite{GARSIA1980229}, and symmetric functions \cite[Section 7.19]{StanEC2}. The descent and inversion statistics on $\mathfrak{S}_n$ also generalize to descent and length statistics on any Coxeter group \cite{BB}, which contain the signed symmetric groups $B_n$ as special cases. Many results on these statistics are known. See \cite{baxter2011,Petersen2015,GARSIA1979288} for some examples on $\mathfrak{S}_n$ and \cite{AdinBrentiRoichman,AdinGesselRoichman,MENDES2008,ReinerEuropJC1993-1} for some examples on $B_n$.

In this paper, we consider statistics defined over the colored permutation groups, which are wreath products of the form $\mathfrak{S}_{n,r}=\mathbb{Z}_r\wr \mathfrak{S}_n$. Colored permutation groups play an essential role in the classification of complex reflection groups \cite{shephard_todd_1954}, and they contain the symmetric groups $\mathfrak{S}_n\cong \mathfrak{S}_{n,1}$ and the signed symmetric groups $B_n\cong \mathfrak{S}_{n,2}$ as special cases. Similar to how elements of $B_n$ can be viewed as certain bijections on $\{\pm 1,\ldots,\pm n\}$, elements in $\mathfrak{S}_{n,r}$ can be viewed as certain bijections on $r$ copies of $\{1,\ldots,n\}$, each indexed with an element in $\mathbb{Z}_r$.

Many statistics on $\mathfrak{S}_{n,r}$ have been studied, and many of these generalize ones on $\mathfrak{S}_n$ and $B_n$. See \cite{steingrim94} and \cite{fire2005statistics} for numerous examples. 
We will focus on the descent and flag major index statistics on $\mathfrak{S}_{n,r}$, which respectively generalize the descent and major index statistics on $\mathfrak{S}_{n}$. The descent statistic $\des_{n,r}$ on $\mathfrak{S}_{n,r}$ was introduced by Steingr\'{i}msson \cite{steingrim94}, who showed that $\des_{n,r}$ is equidistributed with the excedance statistic on $\mathfrak{S}_{n,r}$, and 
its generating function satisfies
\begin{equation}\label{eq:des_distribution}
    \frac{1}{(1-q)^{n+1}}\sum_{(\omega,\tau)\in \mathfrak{S}_{n,r}} q^{\des_{n,r}(\omega,\tau)}=\sum_{i=0}^{\infty} (ir+1)^n q^i.
\end{equation}
The flag major index statistic $\fmaj_{n,r}$ was introduced by Adin and Roichman \cite{AdinRoichman1}, who showed that $\fmaj_{n,2}$ on the signed symmetric group $B_n\cong \mathfrak{S}_{n,2}$ is equidistributed with the length statistic on $B_n$. This is an analog of MacMahon's result involving equidistribution of the major index and inversion statistics on $\mathfrak{S}_n$. Subsequent work by Haglund, Loehr, and Remmel \cite{HAGLUND2005835} established that the distribution of $\fmaj_{n,r}$ for general $r$ is given by
\begin{equation}\label{eq:fmaj_distribution}
    \sum_{(\omega,\tau)\in \mathfrak{S}_{n,r}}q^{\fmaj_{n,r}(\omega,\tau)} = [r]_q [2r]_q\cdots [nr]_q,
\end{equation}
where $[ir]_q=1+q+q^2+\cdots +q^{ir-1}$ is the $q$-integer of $ir$. This coincides with the Poinc\'{a}re polynomial of $\mathfrak{S}_{n,r}$ as a complex reflection group \cite[Theorem 1.4 and Table 1]{geek}, but does not in general coincide with the generating function for the length statistic on $\mathfrak{S}_{n,r}$ \cite[Theorem 4.4]{bagno}.

\subsection*{Main results}

We study the statistics $\des_{n,r}$ and $\fmaj_{n,r}$ on conjugacy classes of $\mathfrak{S}_{n,r}$ with sufficiently long cycles. Recall that a conjugacy class in $\mathfrak{S}_n$ is uniquely determined by the common cycle type of the permutations in the class. Elements in $\mathfrak{S}_{n,r}$ can also be expressed in cycle notation, and this leads to a generalized notion of cycle type that determines conjugacy classes of $\mathfrak{S}_{n,r}$. The precise definition is somewhat technical, so we will reserve a careful treatment for \cref{sec:preliminaries}. Similar to the use of $C_{\lambda}$ for conjugacy classes of $\mathfrak{S}_n$, we use $C_{\rpart}$ to denote the conjugacy classes of $\mathfrak{S}_{n,r}$ indexed by $\rpart$. 

Though there is some prior work involving statistics on conjugacy classes of $\mathfrak{S}_n$ \cite{FulmanJCTA1998,GesselReutenauer,Brenti1993,CooperJonesZhuang2020} and $B_n\cong \mathfrak{S}_{n,2}$ \cite{ReinerEuropJC1993-2,FulmanKimLeePetersen2021}, statistics on conjugacy classes of general colored permutation groups have not been explored heavily. The main theoretical advance appears in recent work by Campion Loth, Levet, Liu, Sundaram, and Yin, where they showed in \cite[Theorem 1.1]{GRWCColoredPermutationStatistics} that any fixed moment of a colored permutation statistic will coincide on all conjugacy classes of $\mathfrak{S}_{n,r}$ with sufficiently long cycles. Our main result strengthens this for the special cases of $\des_{n,r}$ and $\fmaj_{n,r}$.

\begin{theorem}\label{thm:1}
Let $C_{\rpart}$ be a conjugacy class of $\mathfrak{S}_{n,r}$. If $C_{\rpart}$ has no cycles of lengths $1,2,\ldots,2k$, then the $k$-th moments of $\des_{n,r}$ and $\fmaj_{n,r}$ on $C_{\rpart}$ match the respective $k$-th moments on $\mathfrak{S}_{n,r}$.
\end{theorem}

The descent and flag major index statistics are known to be asymptotically normal on $\mathfrak{S}_{n,r}$ \cite{CM2012}. Combining this fact with the Method of Moments and \cref{thm:1}, we obtain the following corollary, which shows asymptotic normality of $\des_{n,r}$ and $\fmaj_{n,r}$ on conjugacy classes with sufficiently long cycles.

\begin{corollary}\label{cor:1}
For every $n\geq 1$, let $C_{\rpart_n}$ be a conjugacy class of $\mathfrak{S}_{n,r}$ such that for all $i$, the number of cycles of length $i$ in $\rpart_n$ approaches 0 as $n\to\infty$. Let $\stat_n$ for $n\geq 1$ be either the descent or flag major index statistic on $C_{\rpart_n}$ with mean $\mu_n$ and variance $\sigma_n^2$. Then as $n\to\infty$, the random variable  $\frac{\stat_n-\mu_n}{\sigma_n}$ converges in distribution to the standard normal distribution. 
\end{corollary}

\subsection*{Related Work}
\cref{thm:1} and \cref{cor:1} were inspired by prior work of Fulman \cite{FulmanJCTA1998}, who proved the analogous results for the descent and major index statistics on conjugacy classes of $\mathfrak{S}_n$ satisfying appropriate conditions on cycle lengths. Our method of proving \cref{thm:1} for $\des_{n,r}$ is based on Fulman's original approach. For $\fmaj_{n,r}$,  we combine our work for $\des_{n,r}$ with the theory of degree for a colored permutation statistic, as introduced in \cite{GRWCColoredPermutationStatistics}. Since the descent and flag major index statistics on $\mathfrak{S}_{n,1}$ coincide with the classical descent and major index statistics on $\mathfrak{S}_n$, specializing \cref{thm:1} and \cref{cor:1} to $\mathfrak{S}_n$ recovers the original results of Fulman.

In the case of $r=2$, the descent statistic $\des_{n,2}$ on $\mathfrak{S}_{n,2}$ does not coincide with the descent statistic $\des_{B_n}$ on $B_n$ as a Coxeter group under the usual isomorphism between $\mathfrak{S}_{n,2}$ and $B_n$. See \cite[Section 8.1]{BB} for a thorough discussion of $\des_{B_n}$. Though $\des_{n,2}$ and $\des_{B_n}$ do not coincide, \cite[Theorem 3.4]{brenti94} and \eqref{eq:des_distribution} show that these statistics share the same distribution on $B_n\cong \mathfrak{S}_{n,2}$. However, one can find conjugacy classes where $\des_{n,2}$ and $\des_{B_n}$ do not share the same distribution. Consequently, \cref{thm:1} and \cref{cor:1} do not apply to $\des_{B_n}$, though analogs of these results for $\des_{B_n}$ were previously established in \cite{GRWCColoredPermutationStatistics}, where the authors also derived explicit formulas for the distribution of $\des_{B_n}$ on conjugacy classes of $B_n$.

\subsection*{Outline of Paper}

We start in \cref{sec:preliminaries} by outlining preliminary information involving colored permutation groups and statistics. We then establish \cref{thm:1} and \cref{cor:1} for the descent and flag major index statistics in \cref{sec:descents} and \cref{sec:flag-major}, respectively. We conclude with potential future directions in \cref{sec:conclusion}. 

\section{Preliminaries}\label{sec:preliminaries}

We begin with preliminaries on the colored permutation groups $\mathfrak{S}_{n,r}$, their conjugacy classes, and the specific statistics considered in this paper. Our definitions are primarily based on what is given in \cite{steingrim94} and \cite{CM2012}. For properties of the conjugacy classes of $\mathfrak{S}_{n,r}$, we use \cite{JamesKerber1981} as a reference, which contains a more general treatment of wreath products. 

\subsection{Colored permutation groups and statistics}
Let $\mathbb{Z}_r$ be the group of integers modulo $r$ and $\mathfrak{S}_n$ be the symmetric group on $[n]=\{1,2,\ldots,n\}$. The \emph{colored permutation group} $\mathfrak{S}_{n,r}$ is the wreath product $\mathbb{Z}_r\wr \mathfrak{S}_n$, which is the semidirect product $\mathbb{Z}_r^n \rtimes \mathfrak{S}_n$ formed from the permutation action of $\mathfrak{S}_n$ on $\mathbb{Z}_r^n$. An element in $\mathfrak{S}_{n,r}$ is called a \emph{colored permutation}, and it will be denoted $(\omega,\tau)$, where $\omega\in \mathfrak{S}_n$ and $\tau: [n]\to \mathbb{Z}_r$ is a function referred to as a \emph{coloring}. For brevity, we will usually express $\tau$ in the form $(\tau(1),\ldots \tau(n))$. From its construction as a wreath product, the group operation is defined as \[(\omega_1,\tau_1)(\omega_2,\tau_2)=(\omega_1\omega_2,(\tau_1\circ \omega_2)+\tau_2).\]

The colored permutation group $\mathfrak{S}_{n,r}$ can be embedded as a subgroup of the symmetric group $\mathfrak{S}_{rn}$, which we describe explicitly as follows. 
Let $[n]^r$ denote the set of $rn$ elements \[\{i^c : i\in [n], \, c\in \mathbb{Z}_r\},\] 
    where the superscript indicates the \emph{color} of an element in $[n]$. One can view the colored permutation $(\omega, \tau)$ as a bijection on $[n]^r$. We abuse notation and also denote this bijection $(\omega,\tau)$, and it is defined by $(\omega,\tau)(i^c) = \omega(i)^{\tau(i)+c}$ for all $i\in [n]$ and $c\in \mathbb{Z}_r$. Observe that for all $i\in [n]$, we have that
    \begin{equation}
        \begin{split}
            (\omega_1,\tau_1)(\omega_2,\tau_2)(i^0) & = (\omega_1,\tau_1)(\omega_2(i)^{\tau_2(i)}) \\
            & =\omega_1\omega_2(i)^{\tau_1(\omega_2(i))+\tau_2(i)}\\
            & =(\omega_1\omega_2,(\tau_1\circ \omega_2)+\tau_2)(i^0),
        \end{split}
    \end{equation}
    so this identification is a group homomorphism that identifies $\mathfrak{S}_{n,r}$ as a subgroup of $\mathfrak{S}_{nr}$. 
    
    Since the images of $i^0$ for $i\in [n]$ are sufficient for determining $(\omega,\tau)\in \mathfrak{S}_{n,r}$, one can use these to form the two-line and one-line notations of $(\omega,\tau)$. The \emph{two-line notation} of $(\omega,\tau)$ is a $2\times n$ array with $1^0,2^0,\ldots n^0$ in the first line and $(\omega,\tau)(1^0),(\omega,\tau)(2^0),\ldots,(\omega,\tau)(n^0)$ in the second line. The \emph{one-line notation} of $(\omega,\tau)$ results from deleting the first line of the two-line notation. We illustrate these with an example below.

\begin{example}\label{perm_example_1a}
    Consider $\omega=[3,8,5,6,2,1,4,7]\in \mathfrak{S}_8$ expressed in one-line notation and the coloring $\tau=(1,0,0,1,2,2,0,1)\in \mathbb{Z}_3^8$. This defines an element in $\mathfrak{S}_{8,3}$ whose corresponding bijection is determined by
    \[(\omega,\tau)(1^0)=3^1, \, (\omega,\tau)(2^0)=8^0, \, (\omega,\tau)(3^0)=5^0, \, (\omega,\tau)(4^0)=6^1 \]
    \[(\omega,\tau)(5^0)=2^2,\, (\omega,\tau)(6^0)=1^2,\, (\omega,\tau)(7^0)=4^0,\, (\omega,\tau)(8^0)=7^1.\]
    This can be expressed in the two-line and one-line notations as 
    \begin{equation}
        \begin{split}
            (\omega,\tau) & =\begin{bmatrix}
        1^0 & 2^0 & 3^0 & 4^0 & 5^0 & 6^0 & 7^0 & 8^0 \\
        3^1 & 8^0 & 5^0 & 6^1 & 2^2 & 1^2 & 4^0 & 7^1
    \end{bmatrix} \\
    & =\begin{bmatrix}
        3^1 & 8^0 & 5^0 & 6^1 & 2^2 & 1^2 & 4^0 & 7^1
    \end{bmatrix}.
        \end{split}
    \end{equation}
\end{example}

We will primarily focus on three statistics on $\mathfrak{S}_{n,r}$: descents, major index, and flag major index. For any $(\omega,\tau)\in \mathfrak{S}_{n,r}$, an index $i\in [n]$ is a \emph{descent} of $(\omega,\tau)$ if $\tau(i)>\tau(i+1)$, or $\tau(i)=\tau(i+1)$ and $\omega(i)>\omega(i+1)$, where we use the convention $\tau(n+1)=0$ and $\omega(n+1)=n+1$. One can alternatively fix the total order on $[n]^r$
\begin{equation}\label{eq:descentordering}
    1^0<2^0<3^0<\cdots <1^1<2^1<3^1<\cdots < 1^{r-1}<2^{r-1}< 3^{r-1}<\cdots
\end{equation}
and define a descent to be any $i\in [n]$ such that $(\omega,\tau)(i^0)>(\omega,\tau)((i+1)^0)$, with the convention that $(\omega,\tau)((n+1)^0)=(n+1)^0$. 

Letting $\Des_{n,r}(\omega,\tau)$ denote the set of descents of $(\omega,\tau)\in \mathfrak{S}_{n,r}$, the descent and major index statistics on $\mathfrak{S}_{n,r}$ are respectively defined as
\[\des_{n,r}(\omega,\tau)=|\Des_{n,r}(\omega,\tau)| \text{ \quad and \quad }\maj_{n,r}(\omega,\tau) = \sum_{i\in \Des_{n,r}(\omega,\tau)\cap [n-1]} i.\]
Observe that when $r=1$, these reduce to the usual descent and major index statistics on $\mathfrak{S}_n$. In this case, we will omit $r$ from the subscript and denote these statistics as $\des_n$ and $\maj_n$. 

The \emph{color} and \emph{flag major index} statistics on $\mathfrak{S}_{n,r}$ are the nonnegative integers defined by 
\[\col_{n,r}(\omega,\tau)=\sum_{i=1}^n \tau(i) \text{ \quad and \quad }\fmaj_{n,r}(\omega,\tau)=r\cdot \maj_{n,r}(\omega,\tau)+\col_{n,r}(\omega,\tau).\]
Note that the $\col_{n,r}$ statistic uses $\{0,1,\ldots,r-1\}$ as representative elements in $\mathbb{Z}_r$ and adds them as elements in $\mathbb{Z}$. In the case when $r=1$, the color statistic is identically 0, so the flag major index coincides with the usual major index on $\mathfrak{S}_n$.

\begin{example}\label{ex:perm_example_1b}
    Consider the permutation $(\omega,\tau)\in \mathfrak{S}_{8,3}$ with one-line notation
    \[(\omega,\tau)=[3^1,8^0,5^0,6^1,2^2,1^2,4^0,7^1]\]
    The descent set of $(\omega,\tau)$ is $\{1,2,5,6,8\}$, and the sum of the colors that appear is $7$. Using this, we calculate \[\des_{8,3}(\omega,\tau)=5, \quad \maj_{8,3}(\omega,\tau)=14, \quad  \text{ and } \quad \fmaj_{8,3}(\omega,\tau)=3\cdot 14+7=49.\]
\end{example}

For any statistic $X:\mathfrak{S}_{n,r} \to \mathbb{R}$, we can consider it as a random variable by equipping $\mathfrak{S}_{n,r}$ with the uniform distribution. The corresponding probability distribution is defined by 
\[ \pr_{\mathfrak{S}_{n,r}}[X=i]=|X^{-1}(i)|/|\mathfrak{S}_{n,r}|. \]
For each positive integer $k$, the $k$-th moment of $X$ will be denoted $\E_{\mathfrak{S}_{n,r}}[X^k]$. For the descent and flag major index statistics, Chow and Mansour established the following results involving their asymptotic distributions. 

\begin{theorem}\cite[Theorem 3.1 and Theorem 3.4]{CM2012}\label{thm:normal_Snr} For any positive integers $n$ and $r$, $\des_{n,r}$ has mean $\mu_{n,r}=\frac{rn+r-2}{2r}$ and variance $\sigma_{n,r}^2=\frac{n+1}{12}$, and as $n\to\infty$, the standardized random variable $\frac{\des_{n,r}-\mu_{n,r}}{\sigma_{n,r}}$ converges to a standard normal distribution.
\end{theorem}

\begin{theorem}\cite[Theorem 4.1 and Theorem 4.3]{CM2012}\label{thm:fmaj_normal_Snr} 
For any positive integers $n$ and $r$, $\fmaj_{n,r}$ has mean $\mu_{n,r}=\frac{n(rn+r-2)}{4}$ and variance $\sigma_{n,r}^2=\frac{2r^2n^3+3r^2n^2+(r^2-6)n}{72}$, and as $n\to\infty$, the standardized random variable $\frac{\fmaj_{n,r}-\mu_{n,r}}{\sigma_{n,r}}$ converges to a standard normal distribution.
\end{theorem}

For our results on the asymptotic distributions of $\des_{n,r}$ and $\fmaj_{n,r}$, we will need additional tools from probability theory. In general, two different probability distributions can share the same moments. We will be primarily interested in normal distributions, which are uniquely determined by their moments, so once we have that the moments of a random variable $X$ coincide with those of a normal distribution, we can conclude that the distribution of $X$ coincides with a normal distribution. We will use this property for normal distributions in conjunction with a tool called the Method of Moments. See \cite[Section 30]{billingsley} for further details of this result.

\begin{theorem}[Method of Moments]
Suppose $\{X_n\}_{n\geq 1}$ and $Y$ are real-valued random variables with finite $k$-th moments for all $k$. If $Y$ is uniquely determined by its moments and 
    \[\lim_{n\to\infty} \E[X_n^k] = \E[Y^k],\]
    for all $k$, then $X_n$ converges in distribution to $Y$.
\end{theorem}

\subsection{Conjugacy classes of colored permutation groups}

Our work will focus on conjugacy classes of $\mathfrak{S}_{n,r}$, which we now describe. Similar to permutations in $\mathfrak{S}_n$, colored permutations also have a cycle notation. Starting with $(\omega,\tau)$, one can express $\omega$ in the usual cycle notation with color $0$ on all elements and then insert $\omega(i)^{\tau(i)}$ under $i^0$ for each $i\in [n]$. We will refer to this as the \emph{two-line cycle notation}. Removing the first row in every cycle then results in the \emph{cycle notation} for $(\omega,\tau)$. An example is shown below.

\begin{example}\label{perm_example_1c}
    Consider the permutation
    \[(\omega,\tau)=[3^1,8^0,5^0,6^1,2^2,1^2,4^0,7^1]\in \mathfrak{S}_{8,3}\]
    given in \cref{perm_example_1a}. The two-line cycle notation is given by
    \[(\omega,\tau)=\begin{pmatrix}
        1^0 & 3^0 & 5^0 & 2^0 & 8^0 & 7^0 & 4^0 & 6^0 \\
        3^1 & 5^0 & 2^2  & 8^0&  7^1& 4^0 &6^1 & 1^2 
    \end{pmatrix}.\]
    Deleting the first line results in the cycle notation
    \[(\omega,\tau)=(3^1 5^0 2^2 8^0 7^1 4^0 6^1 1^2).\]
    As in $\mathfrak{S}_{n}$, the cycle notation for elements in $\mathfrak{S}_{n,r}$ is not unique. 
\end{example}

Similar to permutations in $\mathfrak{S}_n$, colored permutations have a notion of cycle type derived from the cycle notation. An \emph{$r$-partition} of $n\in \mathbb{Z}_+$ is an $r$-tuple of partitions $\rpart=(\lambda^j)_{j=0}^{r-1}$ where each $\lambda^j$ is a partition of some nonnegative integer $n_j$ such that $\sum_{j=0}^{r-1} n_j=n$. For any cycle in the cycle notation of $(\omega,\tau)\in \mathfrak{S}_{n,r}$, its \emph{length} is the number of elements in it, and its \emph{color} is the sum of the colors that appear (as an element in $\mathbb{Z}_r$). The \emph{cycle type} of $(\omega,\tau)\in \mathfrak{S}_{n,r}$ is the $r$-partition $\rpart$ where $\lambda^j$ records the cycle lengths for the cycles with color $j$. 

\begin{example}
    The permutation from \cref{perm_example_1c} has a single cycle of length $8$ with color $7\equiv 1\bmod 3$. Hence, its cycle type is $\rpart=(\emptyset,(8),\emptyset)$.
\end{example}

\begin{example}
    For a larger example, consider the following colored permutation in $\mathfrak{S}_{9,3}$:
    \[(\omega,\tau)=(1^0 3^2 7^1 6^0 )(2^1)(4^2 5^0)(8^0)(9^1).\]
    Since $r=3$, the cycle type of this colored permutation is 
    \[\rpart= (\lambda^{0},\lambda^{1},\lambda^{2})=((1,4),(1^2),(2)),\]
    where each partition has been expressed in multiplicative notion $(1^{a_1},2^{a_2},\ldots,n^{a_n})$.
\end{example}

As in $\mathfrak{S}_n$, the conjugacy classes of $\mathfrak{S}_{n,r}$ are determined by cycle type. 

\begin{proposition}\cite[Theorem~4.2.8 \and Lemmas~4.2.9-4.2.10]{JamesKerber1981} \label{cycletype}
    Two elements $(\omega,\tau),(\omega',\tau')\in \mathfrak{S}_{n,r}$ are conjugate if and only if they share the same cycle type.  Hence, each conjugacy class of $\mathfrak{S}_{n,r}$ is indexed by an $r$-partition. 
\end{proposition}

Throughout, we use $C_{\rpart}$ for the conjugacy class corresponding to colored permutations with cycle type $\rpart$. For a statistic $X$ on $\mathfrak{S}_{n,r}$, we can restrict $X$ to $C_{\rpart}$ and equip $C_{\rpart}$ with the uniform distribution to consider $X$ as a random variable. $X$ then has a discrete probability distribution given by 
\[\pr_{C_{\rpart}}[X=i]=|X^{-1}(i)\cap C_{\rpart}|/|C_{\rpart}|.\]
Note that this is equivalent to the conditional distribution $\pr_{\mathfrak{S}_{n,r}}[X=i\mid C_{\rpart}]$, and the above notation is introduced for brevity. We will also sometimes consider more general sets $\Omega\subseteq \mathfrak{S}_{n,r}$, and we similarly use the notation $\pr_{\Omega}[X=i]$ for the distribution of $X$ on $\Omega$ equipped with the uniform distribution. 

\subsection{Statistics on conjugacy classes with sufficiently long cycles}\label{sec:preliminaries3}

\cite{GRWCColoredPermutationStatistics} analyzes moments of statistics on conjugacy classes of $\mathfrak{S}_{n,r}$ with sufficiently long cycles. We will describe the parts of this work relevant to our results and refer the reader to \cite{GRWCColoredPermutationStatistics} for a detailed account. See also \cite{hamaker2022characters} for further details specific to the symmetric group $\mathfrak{S}_n$. 

A \emph{partial colored permutation} on $\mathfrak{S}_{n,r}$ is a pair $(K,\kappa)$ where $K=\{(i_h,j_h)\}_{h=1}^m$ consists of distinct ordered pairs of elements in $[n]$ and $\kappa:\{i_1,\ldots,i_m\}\to \mathbb{Z}_r$ is any function, which we can represent as ordered pairs $\{(i_h,\kappa(i_h))\}_{h=1}^m$. We call $m$ the \emph{size} of $(K,\kappa)$, and also denote this as $|(K,\kappa)|$. For brevity, we will also express $(K,\kappa)$ using a single set of ordered pairs of elements in $[n]^r$ as \[(K,\kappa)=\left\{\left(i_h^0,j_h^{\kappa(i_h)}\right)\right\}_{h=1}^m.\]
Indeed, the correspondence between these notations is clear. 

\begin{remark}
        In \cite{GRWCColoredPermutationStatistics}, the authors use the convention that $(\omega,\tau)(i^0)=\omega(i)^{\tau(\omega(i))}$ rather than $(\omega,\tau)(i^0)=\omega(i)^{\tau(i)}$. Our alternative convention in this paper simplifies our proofs significantly. We have adapted  \cite{GRWCColoredPermutationStatistics} appropriately to account for this differing convention. In particular, \cite{GRWCColoredPermutationStatistics} uses the convention that $\kappa$ is a function on $\{j_1,\ldots,j_m\}$, while we instead define $\kappa$ on $\{i_1,\ldots,i_m\}$.
\end{remark}

A permutation $\omega\in \mathfrak{S}_n$ \emph{satisfies} $K$ if $\omega(i_h)=j_h$ for all $h\in [m]$. A coloring $\tau:[n]\to \mathbb{Z}_r$ \emph{satisfies} $\kappa$ if $\tau(i_h)=\kappa(i_h)$ for all $h\in [m]$. A colored permutation $(\omega,\tau)\in \mathfrak{S}_{n,r}$ \emph{satisfies} $(K,\kappa)$ if $\omega$ satisfies $K$ and $\tau$ satisfies $\kappa$. Viewing $(\omega,\tau)$ as a bijection on $[n]^r$, this is equivalent to $(\omega,\tau)$ mapping $i_h^0$ to $j_h^{\kappa(i_h)}$ for all $h\in [m]$. We use $I_{(K,\kappa)}:\mathfrak{S}_{n,r}\to \{0,1\}$ to denote the indicator function for a colored permutation satisfying $(K,\kappa)$. In general, the probability of satisfying $(K,\kappa)$ on a conjugacy class $C_{\rpart}\subseteq \mathfrak{S}_{n,r}$ can be difficult to compute. However, when $C_{\rpart}$ has all cycles of sufficiently long length, this probability is well-understood. 

    \begin{lemma}\cite[Lemma 3.11 and Corollary 3.12]{GRWCColoredPermutationStatistics}\label{lem:satisfies}
        Let $(K,\kappa)$ be a partial colored permutation statistic on $\mathfrak{S}_{n,r}$ of size $m$, and let $C_{\rpart}$ be a conjugacy class of $\mathfrak{S}_{n,r}$ without cycles of lengths $1,2,\ldots,m$. Then
        \[\pr_{C_{\rpart}}[\omega \text{ satisfies } K]=\frac{1}{(n-1)(n-2)\cdots (n-m)},\]
        \[\pr_{C_{\rpart}}[\tau \text{ satisfies }\kappa]= \frac{1}{r^m},\]
        \[\pr_{C_{\rpart}}[(\omega,\tau) \text{ satisfies } (K,\kappa)]=\frac{1}{(n-1)(n-2)\cdots (n-m)}\cdot \frac{1}{r^m}.\]
    In particular, satisfying $K$ and $\kappa$ are independent. 
    \end{lemma}

One key insight of \cite{GRWCColoredPermutationStatistics} is that the indicator functions $I_{(K,\kappa)}$ for a partial colored permutations can be viewed as building blocks for colored permutation statistics. Formally, a colored permutation statistic $X:\mathfrak{S}_{n,r}\to \mathbb{R}$ has \emph{degree} $m$ if it is in the $\mathbb{R}$-vector space spanned by $\{I_{(K,\kappa)}: |(K,\kappa)|\leq m\}$ and not in the $\mathbb{R}$-vector space spanned by $\{I_{(K,\kappa)}: |(K,\kappa)|\leq m-1\}$. We give examples below using the statistics relevant to this paper. 

\begin{example}\label{ex:des-maj-fmaj}
    The descent, major index, and flag major index statistics on $\mathfrak{S}_{n,r}$ have degrees at most $2$, as
    \[\des_{n,r}=\sum_{i=1}^{n-1}\sum_{j_1^{c_1} < j_2^{c_2}} I_{\{(i^0,\,j_2^{c_2}),((i+1)^0,\,j_1^{c_1})\}} + \sum_{j=1}^n \sum_{c=1}^{r-1} I_{\{(n^0,\,j^c)\}},\]
    \[\maj_{n,r}=\sum_{i=1}^{n-1}\sum_{j_1^{c_1}< j_2^{c_2}} i\cdot  I_{\{(i^0,\,j_2^{c_2}),((i+1)^0,\,j_1^{c_1})\}},\]
    \[\fmaj_{n,r}=r\cdot \sum_{i=1}^{n-1}\sum_{j_1^{c_1} < j_2^{c_2}} i\cdot  I_{\{(i^0,\,j_2^{c_2}),((i+1)^0,\,j_1^{c_1})\}} + \sum_{i=1}^n \sum_{j=1}^n \sum_{c=0}^{r-1} c\cdot I_{\{(i^0,\,j^c)\}}.\]
    Note that the condition $j_1^{c_1}<j_2^{c_2}$ is with respect to the total order given in \eqref{eq:descentordering}. One can show that for $n\geq 3$, these statistics have degrees exactly $2$, e.g., see the approach in \cite[Theorem 4.20]{GRWCColoredPermutationStatistics}. However, we will not need this exact value for their degrees.
\end{example}

It is clear that if two statistics have degree at most $m_1$ and $m_2$, respectively, then their sum has degree at most $\max\{m_1,m_2\}$. The corresponding property for products is given below. 

\begin{lemma}\cite[Corollary 3.16]{GRWCColoredPermutationStatistics}\label{lem:HRLem4.2}
    Suppose $X_{1}$ and $X_{2}$ are statistics on $\mathfrak{S}_{n,r}$ with degree at most $m_1$ and $m_2$, respectively. Then $X_1\cdot X_2$ has degree at most $m_1+m_2$. In particular, for any integer $k\geq 1$ such that $m_1k\leq n$, we have that $X_1^k$ has degree at most $km_1$. 
\end{lemma}

\cref{lem:satisfies} with \cref{lem:HRLem4.2} can then be used to show the following result. Note that we will primarily be interested in the application of these results to $\des_{n,r}$, $\maj_{n,r}$, and $\fmaj_{n,r}$, and \cref{ex:des-maj-fmaj} showed that these have degree at most $2$. 

\begin{theorem}\cite[Theorem~1.1]{GRWCColoredPermutationStatistics} \label{thm:mainindependence}
Suppose $X:\mathfrak{S}_{n,r}\to \mathbb{R}$ has degree at most $m$. For any $k\geq 1$, the $k$th moment $\E_{C_{\rpart}}[X^k]$ coincides on all conjugacy classes $C_{\rpart}$ of $\mathfrak{S}_{n,r}$ without cycles of lengths $1,2,\ldots,mk$.  
\end{theorem}

\section{Descents}\label{sec:descents}

In this section, we prove \cref{thm:1} and \cref{cor:1} for $\des_{n,r}$. Our methods also apply to $\maj_{n,r}$, so we include results for this statistic as well. Throughout, we define $X_i$ to be the indicator function for a descent at position $i$, 
\[X_i(\omega,\tau)=\begin{cases}
1 & \text{ if $i\in \Des_{n,r}(\omega,\tau)$} \\0 & \text{ otherwise.}
\end{cases}\]
The descent and major index statistics can then be expressed as
\[\des_{n,r}=\sum_{i=1}^n X_i \text{ \quad  and \quad }\maj_{n,r} = \sum_{i=1}^{n-1}i\cdot X_i.\]

Observe that the above decompositions also allow us to decompose the $k$-th powers of the descent and major index statistics in terms of $X_1,\ldots,X_n$ as
\begin{equation*}
    \des_{n,r}^k = \sum_{a_1,\ldots,a_k\in [n]} X_{a_1}\cdots X_{a_k} \text{ and } \maj_{n,r}^k = \sum_{a_1,\cdots,a_k\in [n-1]} a_1\cdots a_k X_{a_1}\cdots X_{a_k}.
\end{equation*}
Note that the $a_1,\ldots,a_k$ need not be distinct. Since expectation is linear, an understanding of the mean of $X_{a_1}\cdots X_{a_k}$ on $\mathfrak{S}_{n,r}$ or $C_{\rpart}$ informs us of the $k$-th moments of $\des_{n,r}$ and $\maj_{n,r}$ on these sets. 

\subsection{Moments on colored permutation groups}

We begin by deriving explicit formulas for the expectation of $X_{a_1}\ldots X_{a_k}$ on $\mathfrak{S}_{n,r}$. We start with the following definitions based on \cite{FulmanJCTA1998}. Our modifications account for the possibility of a descent at position $n$ in $\mathfrak{S}_{n,r}$, which cannot occur in $\mathfrak{S}_n$.

\begin{definition}
The \emph{Young subgroup} generated by $a_1,\ldots,a_k\in [n]$ is the subgroup of $\mathfrak{S}_n$ generated by the adjacent transpositions 
\[\{(a_i,a_i+1):a_i\in \{a_1,\ldots,a_n\}\setminus \{n\} \} .\]
\end{definition}

\begin{definition}
    Let $J$ be the Young subgroup of $\mathfrak{S}_n$ generated by $a_1,\ldots,a_k\in [n]$. The \emph{blocks} induced by $a_1,\ldots,a_k\in [n]$ are the equivalence classes $\B_1,\ldots,\B_t\subseteq [n]$ generated by the following property: $i,j\in [n]$ are in the same equivalence class if some $\omega\in J$ maps $i$ to $j$. Observe that one can equivalently express \[J=\mathfrak{S}_{\B_1}\times \cdots \times \mathfrak{S}_{\B_t},\] where $\mathfrak{S}_{\B_i}$ is the group of permutations on the elements in $\B_i$. 
\end{definition}

\begin{example}
    The blocks induced by $1,2,4,7\in [8]$ are $\B_1=\{1,2,3\}$, $\B_2=\{4,5\}$, $\B_3=\{6\}$, and $\B_4=\{7,8\}$. Note that the blocks induced by $1,2,4,7,8\in [8]$ will be the same. 
\end{example}

Fulman shows in \cite[Proof of Theorem 3]{FulmanJCTA1998}  that when the blocks induced by $a_1,\ldots,a_k\in [n-1]$ are $\B_1,\ldots,\B_t$, 
\begin{equation}\label{eq:fulman_expectation}
    \E_{\mathfrak{S}_{n}}[X_{a_1} X_{a_2}\cdots X_{a_k}]=\prod_{i=1}^t \frac{1}{|\B_i|!}.
\end{equation}
In $\mathfrak{S}_{n,r}$, we will derive the corresponding formulas for $\E_{\mathfrak{S}_{n, r}}[X_{a_1} X_{a_2}\cdots X_{a_k}]$, and there will be two cases depending on whether or not $a_1,\ldots,a_k$ contains $n$. When $a_1,\ldots,a_k$ does not contain $n$, we show that \cref{eq:fulman_expectation} translates directly to $\mathfrak{S}_{n,r}$.

\begin{lemma}\label{prop:e_descents_no_n}
    Let $a_1,\ldots,a_k\in [n-1]$ with induced blocks $\B_1,\ldots,\B_t$. Then
    \begin{equation}
        \E_{\mathfrak{S}_{n,r}}[X_{a_1} X_{a_2}\cdots X_{a_k}] =\E_{\mathfrak{S}_n}[X_{a_1}X_{a_2}\cdots X_{a_k}]  =\prod_{i=1}^t \frac{1}{|\B_i|!}.
    \end{equation}
\end{lemma}

\begin{proof}
    Let $\mathfrak{S}_n$ act on $\mathfrak{S}_{n,r}$ by permuting entries in the one-line notation.   
    This partitions $\mathfrak{S}_{n,r}$ into orbits based on the elements that appear in the one-line notation.
    Each orbit $\Omega_{c}$ can be indexed by $c=(c_1,\ldots,c_n)$, where $c_i\in \mathbb{Z}_r$ is the color of element $i$ in the one-line notation. Let $f_{c}:\{i^{c_i}\}_{i=1}^n \to [n]$ be the unique order-preserving bijection that maps $\{i^{c_i}\}_{i=1}^n$ with the ordering in \eqref{eq:descentordering} to $[n]$ with the usual ordering. 
    
    This induces a bijection $F_{c}:\Omega_{c}\to \mathfrak{S}_n$ that applies $f_c$ on each element in the one-line notation. For example, in the permutation from Example~\ref{perm_example_1a} where $c=(1,0,0,1,2,2,0,1)$, we have that
     \[F_c([3^1  8^0  5^0  6^1 2^2 1^2  4^0  7^1])=[43258716].\] Since $f_c$ is order-preserving, $F_c$ preserves descents at positions $1,2,\ldots,n-1$. Therefore, for all $(\omega,\tau)\in \Omega_c$, we have that \[X_{a_1}X_{a_2}\cdots X_{a_k}(\omega,\tau)=X_{a_1}X_{a_2}\cdots X_{a_k}(F_{c}(\omega,\tau)).\] As $F_c$ is a bijection, this implies that
     \begin{equation}\label{eq:exp_equal}
        \E_{\mathfrak{S}_{n,r}}[X_{a_1}X_{a_2}\cdots X_{a_k}\mid \Omega_c]=\E_{\mathfrak{S}_{n}}[X_{a_1}X_{a_2}\cdots X_{a_k}].
    \end{equation}
    Equation \eqref{eq:exp_equal} holds for every $\Omega_{c}$, so the Law of Total Expectation implies
    \begin{align*}
        \E_{\mathfrak{S}_{n,r}}[X_{a_1} X_{a_2}\cdots X_{a_k}] & = \sum_{c} \pr_{\mathfrak{S}_{n,r}}[\Omega_{c}]\cdot \E_{\mathfrak{S}_{n,r}}[X_{a_1} X_{a_2}\cdots X_{a_k}\mid \Omega_{c}] \\
        & = \sum_{c} \pr_{\mathfrak{S}_{n,r}}[\Omega_{c}]\cdot \E_{\mathfrak{S}_n}[X_{a_1} X_{a_2}\cdots X_{a_k}] \\
        & =\E_{\mathfrak{S}_n} [X_{a_1}X_{a_2}\cdots X_{a_k}].
    \end{align*}
    The result now follows from \eqref{eq:fulman_expectation}.  
\end{proof}

It remains now to consider products involving $X_n$. We start with the case of random variables corresponding to consecutive indices  $X_{m+1}\cdots X_n$ containing $n$. Based on the ordering used to define descents in \eqref{eq:descentordering}, observe that the subset
\[\{(\omega,\tau)\in \mathfrak{S}_{n,r}\ \mid X_{m+1}\cdots X_n(\omega,\tau)=1\}\] 
is equivalent to
\[\{(\omega,\tau)\in \mathfrak{S}_{n,r}\mid X_{m+1}\cdots X_{n-1}(\omega,\tau)=1 \text{ and } \tau(i)\neq 0 \, \forall \, i>m\}.\]
Using this equivalence, we compute the expectation of $X_{m+1}\cdots X_n$ on $\mathfrak{S}_{n,r}$

\begin{lemma}\label{lem:E_descents_with_n}
    For any $1\leq m< n$, the following holds:
    \begin{equation}\label{eq:E_m_through_n}
    \E_{\mathfrak{S}_{n,r}}[X_{m+1}\cdots X_n]= \left(\frac{r-1}{r}\right)^{n-m}\cdot \frac{1}{(n-m)!}.
    \end{equation}
\end{lemma}

\begin{proof}
We first express 
\begin{equation*}
    \begin{split}
        & \, \E_{\mathfrak{S}_{n,r}}[X_{m+1}\cdots X_n]\\
        =& \, \pr_{\mathfrak{S}_{n,r}}[X_{m+1}\cdots X_n=1]\\ 
        =& \,\pr_{\mathfrak{S}_{n,r}}[\{ \tau(i)\neq 0 \, \forall \, i>m\} \cap \{X_{m+1}\cdots X_{n-1}=1\}] \\
        =& \,\pr_{\mathfrak{S}_{n,r}}[\tau(i)\neq 0 \,\forall \, i>m]\cdot \pr_{\mathfrak{S}_{n,r}}[X_{m+1}\cdots X_{n-1}=1\mid \tau(i)\neq 0 \,\forall \, i>m].
    \end{split}
\end{equation*}
The first term is equal to $((r-1)/r)^{n-m}$, so it suffices to show the second term is $1/(n-m)!$. For this, we let $\mathfrak{S}_{n-m}$ act on the set $\{(\omega,\tau)\in \mathfrak{S}_{n,r} : \tau(i)\neq 0 \, \forall \, i>m\}$ by permuting the last $n-m$ entries in the one-line notation. Under this action, each orbit has size $(n-m)!$, and exactly one element in each orbit has these last $n-m$ elements in descending order. The same argument as in \cref{prop:e_descents_no_n} shows then that 
\[\pr_{\mathfrak{S}_{n,r}}[X_{m+1}\cdots X_{n-1}=1\mid \tau(i)\neq 0 \,\forall \, i>m]=\frac{1}{(n-m)!}.\qedhere\]
\end{proof}

\begin{corollary}\label{cor:indexshift}
    For any positive integers $m$ and $n$,
    \[\E_{\mathfrak{S}_{n,r}}[X_1X_2\cdots X_n] = \E_{\mathfrak{S}_{m+n,r}}[X_{m+1}X_{m+2}\cdots X_{m+n}].\]
\end{corollary}

Finally, we consider the expectation of arbitrary products of the $X_i$ statistics that contain $X_n$. Our approach is to again use an action by a symmetric group of appropriate size.

\begin{lemma}\label{lem:independence}
    For any $a_1,\ldots,a_j\in [m-1]\subseteq [n]$, 
    \begin{equation}
        \begin{split}
            & \, \E_{\mathfrak{S}_{n,r}}[X_{a_1} \cdots X_{a_j} X_{m+1} X_{m+2} \cdots X_n] \\
            =& \, \E_{\mathfrak{S}_{n,r}}[X_{a_1}\cdots X_{a_j}]\cdot \E_{\mathfrak{S}_{n,r}}[X_{m+1}X_{m+2}\cdots X_n].
        \end{split}
    \end{equation}
\end{lemma}
\begin{proof}
    Express
    \begin{equation}\label{eq:independence}
        \begin{split}
            &\, \E_{\mathfrak{S}_{n,r}}[X_{a_1} \cdots X_{a_j} X_{m+1} X_{m+2} \cdots X_n] \\
            =& \, \pr_{\mathfrak{S}_{n,r}}[X_{a_1} \cdots X_{a_j} X_{m+1} X_{m+2} \cdots X_n=1] \\
            =& \, \pr_{\mathfrak{S}_{n,r}}[X_{a_1}\cdots X_{a_j}=1]\cdot\pr_{\mathfrak{S}_{n,r}}[X_{m+1}\cdots X_{n}=1\mid X_{a_1}\cdots X_{a_j}=1]
        \end{split}
    \end{equation}
    The first term is $\E_{\mathfrak{S}_{n,r}}[X_{a_1}\cdots X_{a_j}]$, and the group action argument from \cref{lem:E_descents_with_n} shows that
    \[\pr_{\mathfrak{S}_{n,r}}[X_{m+1}\cdots X_{n}=1\mid X_{a_1}\cdots X_{a_j}=1]= \left(\frac{r-1}{r}\right)^{n-m}\cdot \frac{1}{(n-m)!}.\qedhere \]
\end{proof}

\begin{corollary}\label{prop:e_des_with_n}
    Consider any $a_1,\ldots,a_k\in [n]$ with induced blocks $\B_1,\ldots,\B_t$, where $\B_t$ contains $n$. If $n\in \{a_1,\ldots,a_k\}$, then
    \[\E_{\mathfrak{S}_{n,r}}[X_{a_1}\cdots X_{a_k}] = \left(\frac{r-1}{r}\right)^{|\B_t|}\cdot \prod_{i=1}^t \frac{1}{|\B_i|!}.\]
\end{corollary}

\begin{proof}
    Since $n\in \{a_1,\ldots,a_k\}$, we can express $X_{a_1}\cdots X_{a_k}$ equivalently as \[X_{a_1}\cdots X_{a_{j}} X_{m+1}X_{m+2}\cdots X_n,\] where $a_1,\ldots,a_j\in [m-1]$. By \cref{lem:independence},
    \[\E_{\mathfrak{S}_{n,r}}[X_{a_1}\cdots X_{a_k}]= \E_{\mathfrak{S}_{n,r}}[X_{a_1}\cdots X_{a_{j}}]\cdot \E_{\mathfrak{S}_{n,r}}[X_{m+1}\cdots X_n].\]
    The result follows by applying \cref{prop:e_descents_no_n} and \cref{lem:E_descents_with_n}.
\end{proof}

\subsection{Moments on conjugacy classes without short cycles}

We now consider the expectation of $X_{a_1}\cdots X_{a_k}$ on $C_{\rpart}$ without cycles of lengths $1,2,\ldots,2k$ and establish analogs of \cref{prop:e_descents_no_n} and \cref{prop:e_des_with_n}. Our arguments for \cref{prop:e_descents_no_n} and \cref{prop:e_des_with_n} involved group actions where orbits have exactly one element with \[X_{a_1}\cdots X_{a_k}(\omega,\tau)=1,\] and we will define an appropriate action on $C_{\rpart}$ with the same property.

Fix $a_1,\ldots,a_k\in [n]$, let $\B_1,\ldots,\B_t\subseteq [n]$ be the blocks induced by $a_1,\ldots,a_k$, and let $J=\mathfrak{S}_{\B_1} \times \cdots \times \mathfrak{S}_{\B_t}$ be the Young subgroup of $\mathfrak{S}_n$ generated by $a_1,\ldots,a_k$. Define an action of $J$ on $\mathfrak{S}_{n,r}$ as follows: for all $\pi\in J$ and $(\omega,\tau)\in \mathfrak{S}_{n,r}$,
\begin{equation}\label{eq:J_action}
    \pi\cdot (\omega,\tau) = (\pi,\mathbf{0})(\omega,\tau)(\pi,\mathbf{0})^{-1},
\end{equation}
where $\mathbf{0}$ is the zero coloring. Alternatively, this is the conjugation action of $J$ on $\mathfrak{S}_{n,r}$ after identifying $J$ with the subgroup $J\times \mathbf{0}$. The following result describes orbits under the action given in \eqref{eq:J_action}. 

\begin{lemma}\label{lem:conjugation}
Let $(\omega,\tau)\in \mathfrak{S}_{n,r}$. Let $\pi\in \mathfrak{S}_n$ and $\mathbf{0}$ be the zero coloring. If $(i_1^{c_1},i_2^{c_2},\ldots,i_{\ell}^{c_{\ell}})$ is a cycle in $(\omega,\tau)$, then 
\[(\pi,\mathbf{0})(i_1^{c_1},i_2^{c_2},\ldots,i_{\ell}^{c_{\ell}})(\pi,\mathbf{0})^{-1}=(\pi(i_1)^{c_1},\pi(i_2)^{c_2},\ldots,\pi(i_\ell)^{c_{\ell}})\]
is a cycle in $(\pi,\mathbf{0})(\omega,\tau)(\pi,\mathbf{0})^{-1}$.
\end{lemma}

\begin{proof}
    First, observe that $(\pi,\mathbf{0})^{-1}=(\pi^{-1},\mathbf{0})$. Now for any $i_j$, we consider the image of $\pi(i_j)^0$ under $(\pi,\mathbf{0})(\omega,\tau)(\pi,\mathbf{0})^{-1}$:
    \begin{equation*}
        \begin{split}
            (\pi,\mathbf{0})(\omega,\tau)(\pi,\mathbf{0})^{-1}(\pi(i_j)^0) & =      (\pi,\mathbf{0})(\omega,\tau)(i_j^0)  \\
            & = (\pi,\mathbf{0})(i_{j+1}^{c_{j+1}}) \\
            & = \pi(i_{j+1})^{c_{j+1}},
        \end{split}
    \end{equation*}
    where in the case of $j=\ell$, we replace $j+1$ with $1$. Hence, $\pi(i_{j+1})^{c_{j+1}}$ follows $\pi(i_j)^{c_j}$ in the cycle notation as claimed. 
\end{proof}

\cref{lem:conjugation} implies that the orbit of any $(\omega,\tau)\in \mathfrak{S}_{n,r}$ under the action in \eqref{eq:J_action} consists of colored permutations that can be obtained by fixing a cycle notation of $(\omega,\tau)$ and permuting elements within each block $\B_1,\ldots,\B_t$ without changing the location of colors. On conjugacy classes $C_{\rpart}$ without cycles of lengths $1,2,\ldots,2k$, we will show that these orbits are particularly well-behaved. 

\begin{lemma}\label{lem:CDalgorithm}
    Let $a_1,\ldots,a_k\in [n-1]$ with induced blocks $\B_1,\ldots,\B_t$, and let $J=\mathfrak{S}_{\B_1} \times \cdots \times \mathfrak{S}_{\B_t}$ act on a conjugacy class $C_{\rpart}$ of $\mathfrak{S}_{n,r}$ by \eqref{eq:J_action}. If $C_{\rpart}$ contains no cycles of lengths $1,2,\ldots,2k$, then each orbit under this action has size $|J|=\prod_{i=1}^t |\B_i|!$. Furthermore, there is a unique element in each orbit that has descents at $a_1,\ldots,a_k$.
\end{lemma}

To prove \cref{lem:CDalgorithm}, we will define an algorithm that identifies necessary conditions for descents at $a_1,\ldots,a_k$ to appear and replaces elements in each block $\B_1,\ldots,\B_t$ appropriately. This algorithm will generalize one used by Fulman in \cite[Proof of Theorem 3]{FulmanJCTA1998}. Since our algorithm is very technical, we will start with an example.

\begin{example}
    Consider indices $1,2,4,5\in [9]$ and the 9-cycle with color 2 \[(\omega,\tau)=(1^0 3^1 8^2 5^2 2^0 7^0   4^1   9^0 6^2)\in \mathfrak{S}_{9,3}.\] The blocks induced by $1,2,4,5$ are \[\text{$\B_1=\{1,2,3\}$, $\B_2=\{4,5,6\}$, $\B_3=\{7\}$, $\B_4=\{8\}$ and $\B_5=\{9\}$.}\]  We wish to find an element in the orbit of $(\omega,\tau)$ under the action in \eqref{eq:J_action} that has descents at positions $1,2,4$, and $5$, so we start by replacing elements in the cycle notation with the smallest number in its corresponding block, resulting in 
    \begin{equation}\label{eq:CD1}
        (1^0 1^1 8^2 4^2 1^0 7^0 4^1  9^0 4^2 ).
    \end{equation}

    We must now  find an appropriate way to replace the instances of $1$ and $4$ with elements in the same block. Ignoring colors for the moment, we observe that the elements $7,8$, and $9$ appear exactly once, and they are respectively preceded by $1, 1$, and $4$. Both $1$ and $4$ appear multiple times in \eqref{eq:CD1}, so we can try to use the information involving $7,8$, or $9$ to change this. For simplicity, we choose the largest element $9$, which is preceded by a $4$ in \eqref{eq:CD1}. The elements directly after appearances of $4$'s are $1^0$, $9^0$, and $1^0$. Regardless of how these two appearances of $1$ are replaced with other elements in $\B_1=\{1,2,3\}$, the element $9^0$ will still be the largest. Then for descents at positions $4$ and $5$ to occur, the element $4^0$ must map to $9^0$. Using this information, we next consider
    \begin{equation}\label{eq:CD2}
        (1^0 1^1 8^2 5^2 1^0 7^0 4^1 9^0 5^2 ),
    \end{equation}
    as we have determined the image of $4^0$, but we have not determined the images of $5^0$ or $6^0$. Observe that since $9$ appeared exactly once in \eqref{eq:CD1}, the element preceding it in \eqref{eq:CD2} now appears exactly once. 
    
    Continuing in this manner, $8$ is now the largest element that appears only once but whose preceding element $1$ in \eqref{eq:CD2} appears multiple times. The elements that follow these appearances of $1$ are $1^1$, $8^2$, and $7^0$. We wish for descents at positions $1$ and $2$, and the unique option for this is 
    \begin{equation}
        (2^0 1^1 8^2 5^2 3^0 7^0 4^1 9^0 5^2).
    \end{equation} 
    
    Finally, we consider repeated instances of $5$ to obtain 
    \begin{equation}
        (2^0 1^1 8^2 5^2 3^0 7^0 4^1 9^0 6^2).
    \end{equation}
    Observe that this is in the orbit of $(\omega,\tau)$ under the action in \eqref{eq:J_action}, and it has descents at positions $1,2,4,$ and $5$. 
\end{example}

We now give an algorithm that formalizes the example above. We then use this to establish \cref{lem:CDalgorithm}.

\bigskip

\begin{algorithm}[H]\DontPrintSemicolon
\caption{\algname{ColoredDescents}}
\KwIn{$(\omega,\tau)\in \mathfrak{S}_{n,r}$ without cycles of lengths $1,2,\ldots,2k$; indices $a_1,\ldots,a_k\in [n]$}
\KwOut{a colored permutation $(\omega',\tau')\in \mathfrak{S}_{n,r}$ in the orbit of $(\omega,\tau)$ under \eqref{eq:J_action}}
$\B_1,\ldots,\B_t\coloneqq$ blocks induced by $a_1,\ldots,a_k$ \\
$\sigma_1,\ldots,\sigma_m\coloneqq$ cycles of $(\omega,\tau)$ \\
$\sigma_1',\ldots,\sigma_m'\coloneqq$ cycles obtained by starting with $\sigma_1,\ldots,\sigma_m$ and replacing each $i\in [n]$ with the smallest number from the block that contains it 
\label{CD:startingcycles}\\
\While{$\sigma_1',\ldots,\sigma_m'$ contains repeated integers from $[n]$}
{
$S\coloneqq$ subset of $[n]$ consisting of elements that appear exactly once in $\sigma_1',\ldots,\sigma_m'$ \\
$j \coloneqq$ largest element in $S$ whose preceding element $i$ in $\sigma_1',\ldots,\sigma_m'$ appears multiple times \\ 
$\B\coloneqq$ block containing $i$ \\
$i_1,\ldots,i_{\ell}\coloneqq$ elements in $\sigma_1',\ldots,\sigma_m'$ that are in the block $\B$\\
$j_1^{c_1},\ldots,j_{\ell}^{c_{\ell}}\coloneqq $ elements respectively following $i_1,\ldots,i_{\ell}$ in $\sigma_1',\ldots,\sigma_m'$ \\
$\leq\coloneqq $ partial order on $j_1^{c_1},\ldots,j_{\ell}^{c_{\ell}}$ given by \eqref{eq:descentordering} with repeated elements treated as distinct, incomparable elements \\
$\preceq \, \coloneqq$ partial order on $i_1,\ldots,i_{\ell}$ formed by starting with $\leq$, replacing each $j_h^{c_h}$ with $i_h$, and reversing the relation in $\leq$ \\
$\sigma_1',\ldots,\sigma_m'\coloneqq \sigma_1',\ldots,\sigma_m'$ after replacing instances of $i_1,\ldots,i_{\ell}$ with minimal elements in $\B$ in a manner that respects $\preceq$ \label{CD:last}
}
\Return $\sigma_1',\ldots,\sigma_m'$
\end{algorithm}

\bigskip

    \begin{proof}[Proof of \cref{lem:CDalgorithm}]
    It was shown in \cite[Proof of Theorem 3]{FulmanJCTA1998} that the conjugation action of $J$ on any conjugacy class $C_{\lambda}$ of $\mathfrak{S}_n$ without cycles of lengths $1,2,\ldots,2k$ results in orbits of size $|J|=\prod_{i=1}^t |\B_i|!$. Define  $f:\mathfrak{S}_{n,r}\to \mathfrak{S}_n$ to be the projection $f(\omega,\tau)=\omega$. Combining all of this with \cref{lem:conjugation}, we conclude that for any $(\omega,\tau)\in C_{\rpart}$, 
    \begin{equation}
    \begin{split}
        |J\cdot (\omega,\tau)| & =|\{(\pi,\mathbf{0})(\omega,\tau)(\pi,\mathbf{0})^{-1}: \pi\in J\}| \\
        & \geq |\{f((\pi,\mathbf{0})(\omega,\tau)(\pi,\mathbf{0})^{-1}): \pi\in J\}| \\
        & = |\{\pi \omega\pi^{-1} : \pi\in J\}| \\
        & = |J|.
        \end{split}
    \end{equation}
    Since $|J\cdot (\omega,\tau)|\leq |J|$ always holds, we conclude $|J\cdot (\omega,\tau)|=|J|$. It now suffices to show that there is a unique element in each orbit with descents at $a_1,\ldots,a_k$, which we do using \algname{ColoredDescents}. We start by showing that this algorithm is well-defined.

    First, observe that the $k$ elements in $a_1,\ldots,a_k$ can induce at most $k$ blocks of size larger than $1$, which accounts for at most $2k$ of the elements in $[n]$. Hence, some blocks in $\B_1,\ldots,\B_t$ must initially consist of only one element. If $(\omega,\tau)\in \mathfrak{S}_{n,r}$ has no cycles of lengths $1,2,\ldots,2k$, then each cycle $\sigma_1',\ldots,\sigma_m'$ in line \ref{CD:startingcycles} must contain some element from a block of size $1$. Consequently, choosing $j$ in the \algname{while} loop is well-defined in the first iteration. After each iteration of the \algname{while} loop, the number of elements that appear exactly once increases in at least one cycle in $\sigma_1',\ldots,\sigma_m'$, as the element that precedes $j$ appears multiple times at the start of the loop but appears exactly once at the end of the loop. Consequently, future iterations of the \algname{while} loop are well-defined, and the algorithm will continue until it terminates at a colored permutation. Furthermore, \algname{ColoredDescents} only replaces elements in the cycle notation with others in the same block while leaving colors unchanged, so by \cref{lem:conjugation}, the output of this algorithm is in the same $J$-orbit as the original colored permutation.
    
    It remains to show that the output permutation from \algname{ColoredDescents} is the unique one in the $J$-orbit of $(\omega,\tau)$ with descents at $a_1,\ldots,a_k$. Observe that at the start of the algorithm, the cycles $\sigma_1',\ldots,\sigma_m'$ in \algname{ColoredDescents} trivially satisfy the property that whenever $i_1<i_2$ appear in $\sigma_1',\ldots,\sigma_m'$ and belong to the same block, the elements $j_1^{c_1}$ and $j_2^{c_2}$ that follow them in $\sigma_1',\ldots,\sigma_m'$ satisfy $j_1^{c_1}>j_2^{c_2}$ with respect to the ordering for descents given in \eqref{eq:descentordering}. By the replacement procedure in the algorithm, this property is preserved after each iteration of the \algname{while} loop, so the colored permutation resulting from \algname{ColoredDescents} has the property that when $i_1<i_2$ are in the same block, the elements following them in the cycle notation satisfy $j_1^{c_1}>j_2^{c_2}$. Consequently, the colored permutation resulting from the algorithm has descents at $a_1,\ldots,a_k$. Additionally, it is clear that at each iteration of line \ref{CD:last}, the algorithm identifies necessary conditions for descents to eventually occur at $a_1,\ldots,a_k$, and the replacement used at this line is unique. Consequently, the output colored permutation must be the unique permutation in the orbit of $(\omega,\tau)$ that has descents at $a_1,\ldots,a_k$. 
\end{proof}

\begin{lemma}\label{prop:e_des_conjugacy_class}
    Let $a_1,\ldots,a_k\in [n]$ with induced blocks $\B_1,\ldots,\B_t$, where $\B_t$ contains $n$. Let $C_{\rpart}$ be any conjugacy class of $\mathfrak{S}_{n,r}$ that contains no cycles of lengths $1,2,\ldots,2k$. If $a_1,\ldots,a_k\in [n-1]$, then
    \begin{equation}\label{eq:DCC1}
        \E_{C_{\rpart}}[X_{a_1}X_{a_2}\cdots X_{a_k}]=\prod_{i=1}^t \frac{1}{|\B_i|!}.
    \end{equation}
    Otherwise,
    \begin{equation}\label{eq:DCC2}
        \E_{C_{\rpart}}[X_{a_1}X_{a_2}\cdots X_{a_k}]=\left(\frac{r-1}{r}\right)^{|\B_t|}\cdot \prod_{i=1}^t \frac{1}{|\B_i|!}.
    \end{equation}
\end{lemma}

\begin{proof}
    First consider when $a_1,\ldots,a_k\in [n-1]$. Define $J=\mathfrak{S}_{\B_1}\times \ldots \times \mathfrak{S}_{\B_t}$, and let $\omega\in J$ act on $C_{\rpart}$ by conjugation as in \eqref{eq:J_action}. \cref{lem:CDalgorithm} shows that this action decomposes $C_{\rpart}$ into orbits of size $|J|$ where exactly one element in each orbit has descents at $a_1,\ldots,a_k$. This implies \eqref{eq:DCC1}.

    For \eqref{eq:DCC2}, we assume without loss of generality that $a_k=n$ and $a_1,\ldots,a_{k-1}\in [n-1]$. Expressing $B_t=\{m+1,\ldots,n\}$, we have that
    \begin{equation}\label{eq:descent_with_n}
        \begin{split}
            & \, \E_{C_{\rpart}}[X_{a_1}X_{a_2}\cdots X_{a_k}] \\
            =& \,\pr_{C_{\rpart}}[\tau(i)\neq 0 \, \forall \, i>m]\cdot \pr_{C_{\rpart}}[X_{a_1}\cdots X_{a_{k-1}}=1\mid \tau(i)\neq 0 \, \forall \, i>m].
        \end{split}
    \end{equation}
    By summing over all choices of nonzero colors and applying \cref{lem:satisfies}, the first term is $((r-1)/r)^{n-m}$. For the second term, \cref{lem:conjugation} shows that the action of $J$ preserves the property that $\tau(i)\neq 0$ for all $i>m$, as the colors on the elements following $m+1,\ldots,n$ in the cycle notation are unaffected by the conjugation action of $J$. Hence, this action stabilizes the subset in $C_{\rpart}$ where $\tau(i)\neq 0$ for all $i>m$. \cref{lem:CDalgorithm} then implies that the second term in \eqref{eq:descent_with_n} is $1/|J|$ as needed.
\end{proof}

Combining our results, we can now establish \cref{thm:1} for $\des_{n,r}$. In fact, this result holds for any statistic that is a linear combination of the  statistics $X_i$, including $\maj_{n,r}$. 

\begin{theorem}\label{thm:same_expectation}
    Let $X=\sum_{i=1}^n c_iX_i$ with $c_i\in \mathbb{R}$, and let $C_{\rpart}$ be a conjugacy class of $\mathfrak{S}_{n,r}$. If $C_{\rpart}$ contains no cycles of lengths $1,2,\ldots,2k$, then $\E_{C_{\rpart}}[X^k]=\E_{\mathfrak{S}_{n,r}}[X^k]$. Furthermore, if $c_n=0$, then this is also equal to $\E_{\mathfrak{S}_n}[X^k]$. 
\end{theorem}

\begin{proof}
    Using the decomposition $X=\sum_{i=1}^n c_iX_i$ with $c_i\in \mathbb{R}$ and expanding, we obtain
    \begin{equation}\label{eq:CC_moment}
        \E_{C_{\rpart}}[X^k] = \sum_{a_1,\ldots,a_k\in [n]} \left(\prod_{i=1}^k c_i \right) \cdot \E_{C_{\rpart}}[X_{a_1}\cdots X_{a_k}].
    \end{equation}
    Note that the summation ranges over all possible $a_1,\ldots,a_k$, so it is possible that some of the $X_i$'s in the product $X_{a_1}\cdots X_{a_k}$ are redundant. Regardless, using the fact that $C_{\rpart}$ has no cycles of lengths $1,2,\ldots,2k$ with \cref{prop:e_descents_no_n}, \cref{prop:e_des_with_n}, and \cref{prop:e_des_conjugacy_class}, each of the summands in \eqref{eq:CC_moment} is equal to the corresponding summand in
    \begin{equation}\label{eq:Snr_moment}
        \E_{\mathfrak{S}_{n,r}}[X^k] = \sum_{a_1,\ldots,a_k\in [n]} \left(\prod_{i=1}^k c_i \right) \E_{\mathfrak{S}_{n,r}}[X_{a_1}\cdots X_{a_k}],
    \end{equation}
    so the $k$-th moments of $X$ on $C_{\rpart}$ and $\mathfrak{S}_{n,r}$ coincide. 
    
    In the case where $c_n=0$, we can restrict the summation in \eqref{eq:Snr_moment} to $a_1,\ldots,a_k\in [n-1]$. \cref{prop:e_des_conjugacy_class} then implies that each term in the summation for $\E_{\mathfrak{S}_{n,r}}[X^k]$ equals the corresponding one in
    \begin{equation}
        \E_{\mathfrak{S}_{n}}[X^k] = \sum_{a_1,\ldots,a_k\in [n-1]} \left(\prod_{i=1}^k c_i \right) \E_{\mathfrak{S}_{n}}[X_{a_1}\cdots X_{a_k}]
    \end{equation}
    so the $k$-th moments of $X$ on $\mathfrak{S}_{n,r}$ and $\mathfrak{S}_n$ coincide. 
\end{proof}

\begin{corollary}\label{cor:same_expectation}
    Let $C_{\rpart}$ be a conjugacy class of $\mathfrak{S}_{n,r}$. If $C_{\rpart}$ contains no cycles of lengths $1,2,\ldots,2k$, then
    \[\E_{C_{\rpart}}[\des_{n,r}^k]=\E_{\mathfrak{S}_{n,r}}[\des_{n,r}^k] \text{ \quad  and \quad } \E_{C_{\rpart}}[\maj_{n,r}^k]=\E_{\mathfrak{S}_{n,r}}[\maj_{n,r}^k]=\E_{\mathfrak{S}_n}[\maj_{n}^k].\]
\end{corollary}

We conclude this section with \cref{cor:1} for $\des_{n,r}$. Our proof combines the preceding result with the Method of Moments and known asymptotic results.

\begin{corollary}\label{thm:CLT}
    For every $n\geq 1$, let $C_{\rpart_n}$ be a conjugacy class of $\mathfrak{S}_{n,r}$. Suppose that for all $i$, the number of cycles of length $i$ in $\rpart_n$ approaches 0 as $n\to\infty$. Then for sufficiently large $n$, $\des_{n,r}$ has mean $\mu_{n,r}=\frac{rn+r-2}{2r}$ and variance $\sigma_{n,r}^2=\frac{n+1}{12}$ on $C_{\rpart_n}$, and as $n\to\infty$, the random variable $\frac{\des_{n,r}-\mu_{n,r}}{\sigma_{n,r}}$ converges in distribution to the standard normal distribution. 
\end{corollary}

\begin{proof}
    The mean and variance follow from using \cref{thm:same_expectation} on the first two moments of $\des_{n,r}$ with the assumption that there are no cycles of lengths $1,2,3$, and $4$ for sufficiently large $n$. For the asymptotic behavior, we fix $k$ and express 
    \begin{equation}\label{eq:normalized_expansion}
        \left(\frac{\des_{n,r}-\mu_{n,r}}{\sigma_{n,r}}\right)^k = \frac{1}{\sigma_{n,r}^k}\sum_{i=0}^k \binom{k}{i} (-1)^{k-i} \mu_{n,r}^{k-i}\des_{n,r}^i.
    \end{equation}
    For all sufficiently large $n$, $C_{\rpart_n}$ contains no cycles of lengths $1,2,\ldots,2k$, so \cref{cor:same_expectation} implies that $\E_{\mathfrak{S}_{n,r}}[\des_{n,r}]=\E_{C_{\rpart_n}}[\des_{n,r}]$ when $n$ is sufficiently large. Hence, the expectation of \eqref{eq:normalized_expansion} on $\mathfrak{S}_{n,r}$ and $C_{\rpart_n}$ coincide when $n$ is sufficiently large. Consequently, this equality holds as $n\to\infty$, and the result now follows from the Method of Moments and \cref{thm:normal_Snr}.
\end{proof}

\begin{corollary}\label{thm:CLTmaj}
    For every $n\geq 1$, let $C_{\rpart_n}$ be a conjugacy class of $\mathfrak{S}_{n,r}$. Suppose that for all $i$, the number of cycles of length $i$ in $\rpart_n$ approaches 0 as $n\to\infty$. Then for sufficiently large $n$, $\maj_{n,r}$ has mean $\mu_{n,r}=\frac{n(n-1)}{4}$ and variance $\sigma_{n,r}^2=\frac{n(2n^2+3n-5)}{72}$ on $C_{\rpart_n}$, and as $n\to\infty$, the random variable $\frac{\maj_{n,r}-\mu_{n,r}}{\sigma_{n,r}}$  converges in distribution to the standard normal distribution. 
\end{corollary}

\begin{proof}
    Apply the same argument from \cref{thm:CLT}, but comparing moments of $\frac{\maj_{n,r}-\mu_{n,r}}{\sigma_{n,r}}$ on $C_{\rpart}$ with the moments of $\frac{\maj_n-\mu_{n,1}}{\sigma_{n,1}}$ on $\mathfrak{S}_n$. The asymptotic normality of Mahonian distributions is well-known \cite{feller}.
\end{proof}

\section{Flag major index}\label{sec:flag-major}

In this section, we establish \cref{thm:1} and \cref{cor:1} for the flag major index statistic $\fmaj_{n,r}$. Our general approach follows the one that we used for $\des_{n,r}$ and $\maj_{n,r}$. However, we need several modifications to account for the $\col_{n,r}$ statistic, and our techniques involve the degree of a colored permutation statistic, as described in \cref{sec:preliminaries3}.  

Throughout this section, we define $Y_{i,c}$ to be the indicator function for the color of $i\in [n]$ being $c\in\mathbb{Z}_r$, 
\[Y_{i,c}(\omega,\tau)=\begin{cases}1 & \text{ if $\tau(i)=c$}\\
0 & \text{ otherwise.}\end{cases}\]
Using the same $X_i$ indicator functions for descents, this allows us to express $\fmaj_{n,r}$ as 
\[\fmaj_{n,r}=r\cdot \sum_{i=1}^{n-1} i X_i+\sum_{i=1}^n \sum_{c=0}^{r-1} cY_{i,c}.\]
In particular, $\fmaj_{n,r}^k$ can be expressed as linear combinations of the random variables
\begin{equation}\label{eq:product}
    X_{a_1}\cdots X_{a_j} Y_{a_{j+1},c_{j+1}} \cdots Y_{a_k,c_k}
\end{equation}
where $a_1,\ldots,a_j\in [n-1]$, $a_{j+1},\ldots,a_k\in [n]$, and $c_{j+1},\ldots,c_k\in \mathbb{Z}_r$. We will consider products of this form, and show that their expectations coincide on $\mathfrak{S}_{n,r}$ and all $C_{\rpart}$ without cycles of lengths $1,2,\dots,2k$. We start with a definition and then give a result on the degree of \eqref{eq:product}.

\begin{definition}\label{def:essXYproduct}
Let $a_1,\ldots,a_j\in [n-1]$, $a_{j+1},\ldots,a_k\in [n]$, and $c_{j+1},\ldots,c_k\in \mathbb{Z}_r$. The \emph{essential set} of the statistic $X_{a_1}\cdots X_{a_j}Y_{a_{j+1},c_{j+1}} \cdots Y_{a_k,c_k}$ is \[\Ess(X_{a_1}\cdots X_{a_j}Y_{a_{j+1},c_{j+1}} \cdots Y_{a_k,c_k})=\left(\bigcup_{i=1}^{j} \{a_i,a_{i}+1\}\right)\bigcup \left( \bigcup_{i=j+1}^k \{a_i\}\right) . \]
Each element in $\Ess(X_{a_1}\cdots X_{a_j}Y_{a_{j+1},c_{j+1}} \cdots Y_{a_k,c_k})$ will be called an \emph{essential position}.
\end{definition}

\begin{lemma}\label{lem:XYproducts}
    Let $a_1,\ldots,a_j\in [n-1]$, $a_{j+1},\ldots,a_k\in [n]$, and $c_{j+1},\ldots,c_k\in \mathbb{Z}_r$. Then $Z=X_{a_1}\cdots X_{a_j}Y_{a_{j+1},c_{j+1}} \cdots Y_{a_k,c_k}$ has degree at most $j+k$. Consequently, its mean coincides on all conjugacy classes $C_{\rpart}$ of $\mathfrak{S}_{n,r}$ without cycles of lengths $1,2,\ldots,j+k$. The same holds for $ZY_{i,c}$ when $i\in \Ess(Z)$ and $c\in \mathbb{Z}_r$ is arbitrary. 
\end{lemma}

\begin{proof}
    By \cref{thm:mainindependence}, it suffices to show that $Z$ and $ZY_{i,c}$ have degree at most $j+k$.  Observe that summands for $\fmaj_{n,r}$ in \cref{ex:des-maj-fmaj} can be used to express each $X_{a_i}$ using partial colored permutations of size $2$ and each $Y_{a_i,c}$ using partial colored permutations of size $1$. Using \cref{lem:HRLem4.2}, $Z$ has degree at most $2j+(k-j)=j+k$.
    
    For $ZY_{i,c}$, first observe that the resulting expansion described above for $Z$ consists of linear combinations of statistics of the form
    \begin{equation}\label{eq:XYsizeexpansion}
        \prod_{i=1}^j I_{\{(a_i^0,\,x_i^{c_i}),((a_i+1)^0,\,y_i^{d_i})\}}\cdot \prod_{i=j+1}^n I_{\{(a_i^0,\,z_i^{c_i})\}},
    \end{equation}
    where $x_i^{c_i}>y_i^{d_i}$ and $z_i^{c_i}$ are elements in $[n]^r$. Additionally, we can express
    \begin{equation}\label{eq:extraX}
        Y_{i,c}=\sum_{x=1}^{n} I_{\{(i^0,\,x^c)\}}.
    \end{equation}
    It now suffices to show that the product of \eqref{eq:XYsizeexpansion} with any summand $I_{\{(i^0,\,x^c)\}}$ of \eqref{eq:extraX} has degree at most $j+k$. 
    
    Since $i\in \Ess(Z)$, there is some $I_{(K,\kappa)}$ in the product \eqref{eq:XYsizeexpansion} where $(K,\kappa)$ contains an ordered pair of the form $(i^0,z^d)$. If $x^c=z^d$, then multiplying \eqref{eq:XYsizeexpansion} by $I_{(i^0,x^c)}$ has no effect, and hence, this additional indicator function can be omitted. Otherwise, $x^c\neq z^d$ implies that the product of \eqref{eq:XYsizeexpansion} and $I_{\{(i^0,\,x^c)\}}$ is identically 0. Combined, we conclude that $ZY_{i,c}$ also has degree at most $j+k$.
\end{proof}

In statistics of the form $X_{a_1}\cdots X_{a_j}Y_{a_{j+1},c_{j+1}} \cdots Y_{a_k,c_k}$, some of the elements in $a_{j+1},\ldots,a_k$ may be involved with descents at positions $a_1,\ldots,a_j$, while others are not. Our next result allows us to reduce to when all elements in $a_{j+1},\ldots,a_k$ are involved in descents at $a_1,\ldots,a_j$.

\begin{lemma}\label{lem:fmaj_reduction1}
    Let $a_1,\ldots,a_j\in [n-1]$, $a_{j+1},\ldots,a_k\in [n]$, and $c_{j+1},\ldots,c_k\in \mathbb{Z}_r$. If $a_k\notin \Ess(X_{a_1}\cdots X_{a_j} Y_{a_{j+1},c_{j+1}} \cdots Y_{a_{k-1},c_{k-1}}),$ then
    \begin{equation}
        \begin{split}
            & \, \E_{\mathfrak{S}_{n,r}}[X_{a_1}\cdots X_{a_j} Y_{a_{j+1},c_{j+1}} \cdots Y_{a_{k-1},c_{k-1}}Y_{a_k,c_k}] \\
            = & \, \frac{1}{r}\cdot \E_{\mathfrak{S}_{n,r}}[X_{a_1}\cdots X_{a_j} Y_{a_{j+1},c_{j+1}} \cdots Y_{a_{k-1},c_{k-1}}].
        \end{split}
    \end{equation}
    The same holds on any $C_{\rpart}$ without cycles of lengths $1,2,\ldots,j+k$.
\end{lemma}

\begin{proof}
    For brevity, let $Z=X_{a_1}\cdots X_{a_j} Y_{a_{j+1},c_{j+1}} \cdots Y_{a_{k-1},c_{k-1}}$ and express
    \begin{equation*}
        \begin{split}
            \E_{\mathfrak{S}_{n,r}}[ZY_{a_k,c_k}] & =\pr_{\mathfrak{S}_{n,r}}[ZY_{a_k,c_k}=1] \\
            & =\pr_{\mathfrak{S}_{n,r}}[Z=1]\cdot \pr_{\mathfrak{S}_{n,r}}[Y_{a_k,c_k}=1\mid Z=1] \\
            & =\E_{\mathfrak{S}_{n,r}}[Z]\cdot \pr_{\mathfrak{S}_{n,r}}[Y_{a_k,c_k}=1\mid Z=1].
        \end{split}
    \end{equation*}
    It now suffices to show the second term is $1/r$. Define an action of $\mathbb{Z}_r$ on $\mathfrak{S}_{n,r}$ as follows: $c\in \mathbb{Z}_r$ acts on $(\omega,\tau)$ by adding $c$ to $\tau(a_k)$. Since $a_k$ is not an essential position of $Z$, this group action is stable on the set of elements where $Z=1$. Within each orbit of size $r$, exactly one satisfies $\tau(a_k)=c_k$. Hence, \[\pr_{\mathfrak{S}_{n,r}}[Y_{a_k,c_k}=1\mid Z=1]=1/r\] as desired.

    For conjugacy classes without cycles of lengths $1,2,\ldots,j+k$, we let $\Omega_c$ be the conjugacy class of $\mathfrak{S}_{n,r}$ consisting of permutations with a single cycle of length $n$ and color $c$, and we consider $\Omega=\bigcup_{c\in \mathbb{Z}} \Omega_c$. The same action of $\mathbb{Z}_r$ on $\mathfrak{S}_{n,r}$ given above is stable on $\Omega$, implying 
    \[\E_{\Omega}[ZY_{a_k,c_k}] = \frac{1}{r} \E_{\Omega}[Z].\]
    \cref{lem:XYproducts} with the Law of Total Expectation implies that
    \[\E_{\Omega}[Z]=\sum_{c\in \mathbb{Z}_r}\pr_{\Omega}[\Omega_c]\cdot \E_{\Omega_c}[Z] = \sum_{c\in \mathbb{Z}_r}\frac{1}{r}\cdot \E_{\Omega_0}[Z]= \E_{\Omega_0}[Z],\]
    and the same holds when $Z$ is replaced with $ZY_{i_k,c_k}$. Applying \cref{lem:XYproducts} again allows us to conclude that on any $C_{\rpart}$ without cycles of lengths $1,2,\ldots,j+k$,
    \[\E_{C_{\rpart}}[ZY_{i_k,c_k}]=\E_{\Omega_0}[ZY_{a_k,c_k}]=\E_{\Omega}[ZY_{a_k,c_k}] = \frac{1}{r} \E_{\Omega}[Z] = \frac{1}{r} \E_{\Omega_0}[Z]= \frac{1}{r} \E_{C_{\rpart}}[Z]. \qedhere\]
\end{proof}

We now consider when $a_{j+1},\ldots,a_k$ is the essential set of $X_{a_1}\cdots X_{a_j}$. In this case, our preceding work with \algname{ColoredDescents} shows that the mean coincides on $\mathfrak{S}_{n,r}$ and any conjugacy classes without cycles of length $2j$.

\begin{lemma}\label{lem:fmaj_reduction2}
    Let $a_1,\ldots,a_j\in [n-1]$, $a_{j+1},\ldots,a_k\in [n]$, and $c_{j+1},\ldots,c_k\in \mathbb{Z}_r$, and define $Z=X_{a_1}\cdots X_{a_j}Y_{a_{j+1},c_{j+1}} \cdots Y_{a_k,c_k}$. If \[\Ess(X_{a_1}\cdots X_{a_j})=\Ess(Y_{a_{j+1},c_{j+1}} \cdots Y_{a_k,c_k}),\] then on any conjugacy class $C_{\rpart}$ of $\mathfrak{S}_{n,r}$ without cycles of lengths $1,2,\ldots,2j$,
    \[\E_{\mathfrak{S}_{n,r}}[Z]=\E_{C_{\rpart}}[Z].\]
\end{lemma}

\begin{proof}
    We can assume without loss of generality that all of the elements $a_1,\ldots,a_j$ are distinct, and all of the elements $a_{j+1},\ldots,a_k$ are distinct. Let $\B_1,\ldots,\B_t$ be the blocks induced by $a_1,\ldots,a_j$. We first consider when there exists some $i$ and $i+1$ in the same block, and both $Y_{i,c}$ and $Y_{i+1,c'}$ appear in $Z$ for some $c$ and $c'$. If $c<c'$, then a descent at position $i$ is impossible. This implies $Z=0$ on both the entire group and on conjugacy classes, so the result is clear. If $c>c'$, then removing $X_i$ from the product defining $Z$ results in the same statistic. Iterating this argument, we can assume without loss of generality that for any $i$ and $i+1$ in the same block where some $Y_{i,c}$ and $Y_{i+1,c'}$ appear in $Z$, we have $c=c'$. Combined with the fact that
    \[\Ess(X_{a_1}\cdots X_{a_j})=\Ess(Y_{a_{j+1},c_{j+1}} \cdots Y_{a_k,c_k}),\]
    this implies that the property $Y_{a_{j+1},c_{j+1}} \cdots Y_{a_k,c_k}(\omega,\tau)=1$ is equivalent to $\tau$ satisfying some fixed $\kappa: \{a_{j+1},\ldots,a_k\}\to \mathbb{Z}_r$ that is either constant on or not defined on each block $\B_1,\ldots,\B_t$.

    We now show the claimed equality, first by considering $\mathfrak{S}_{n,r}$. Express $\E_{\mathfrak{S}_{n,r}}[Z]$ as 
    \begin{equation*}
        \begin{split} \,\pr_{\mathfrak{S}_{n,r}}[X_{a_1}\cdots X_{a_j}=1\mid Y_{a_{j+1},c_{j+1}}\cdots Y_{a_k,c_k}=1]\cdot \pr_{\mathfrak{S}_{n,r}}[Y_{a_{j+1},c_{j+1}}\cdots Y_{a_k,c_k}=1]. 
        \end{split}
    \end{equation*}
    This can be rewritten as \begin{equation}\label{eq:specialcase1}\pr_{\mathfrak{S}_{n,r}}[X_{a_1}\cdots X_{a_j}=1\mid \text{$\tau$ satisfies $\kappa$}]\cdot \pr_{\mathfrak{S}_{n,r}}[\text{$\tau$ satisfies $\kappa$}].\end{equation}
    There are $k-j$ elements in the domain of $\kappa$, so the second term in \eqref{eq:specialcase1} is $1/r^{k-j}$. For the first term, we use a similar approach as the one for descents. Let $J=\mathfrak{S}_{\B_1}\times \ldots \times \mathfrak{S}_{\B_t}$ act by permuting the one-line notation within each block so that $\sigma \in \mathfrak{S}_{\B_j}$ permutes the images of $i^0$ for $i\in \B_j$. Since $\kappa$ is constant or undefined on each block, this action stabilizes the subset of colored permutations satisfying $\kappa$. Each orbit has size $|J|$ and contains exactly one element where the one-line notation within each block is in descending order. Hence, exactly one element in each orbit has the appropriate descents at $a_1,\ldots,a_j$, and we conclude
    \begin{equation}\label{eq:specialcase2Snr}
        \E_{\mathfrak{S}_{n,r}}[Z]=\frac{1}{r^{k-j}}\prod_{i=1}^t \frac{1}{|\B_i|}.
    \end{equation}
    
    For a conjugacy class $C_{\rpart}$ without cycles of lengths $1,2,\ldots 2j$, we similarly express 
    \begin{equation}\label{eq:specialcase2}
                \E_{C_{\rpart}}[Z]=\pr_{C_{\rpart}}[X_{a_1}\cdots X_{a_j}=1\mid \text{$\tau$ satisfies $\kappa$}]\cdot \pr_{C_{\rpart}}[\text{$\tau$ satisfies $\kappa$}].
    \end{equation}
    Since we assumed that 
    \[\Ess(X_{a_1}\cdots X_{a_j})=\Ess(Y_{a_{j+1},c_{j+1}} \cdots Y_{a_k,c_k})\]
    and this common essential set has at most $2j$ elements, we conclude that the domain of $\kappa$ has size $k-j\leq 2j$. Since $C_{\rpart}$ has no cycles of lengths $1,2,\ldots,2j$, the second term in \eqref{eq:specialcase2} above is $1/r^{k-j}$ by \cref{lem:satisfies}. For the first term, let $\pi\in J$ act on $(\omega,\tau)$ as conjugation by $(\pi,\mathbf{0})$. If $\kappa$ is constant on a block $B_j$, then any $(\omega,\tau)\in C_{\rpart}$ where $\tau$ satisfies $\kappa$ has the property that the elements following $i\in \B_j$ in the cycle notation of $(\omega,\tau)$ have the same color. Then \cref{lem:conjugation} implies that this property is preserved under the action of $J$, and hence $J$ stabilizes the elements in $C_{\rpart}$ satisfying $\kappa$. This allows us to apply \cref{lem:CDalgorithm} to conclude that exactly one element in each orbit of size $|J|$ has descents at $a_1,\ldots,a_j$. Thus, \eqref{eq:specialcase2} becomes
    \[\E_{C_{\rpart}}[Z]=\frac{1}{r^{k-j}}\prod_{i=1}^t \frac{1}{|\B_i|},\]
    and the right side coincides with the right side of \eqref{eq:specialcase2Snr}.
\end{proof}

Finally, we show that the mean of $X_{a_1}\cdots X_{a_j}Y_{a_{j+1},c_{j+1}} \cdots Y_{a_k,c_k}$ coincides on $\mathfrak{S}_{n,r}$ and appropriate $C_{\rpart}$. We then conclude with \cref{thm:1} and \cref{cor:1} for $\fmaj_{n,r}$. 

\begin{lemma}\label{lem:final_fmaj}
Let $a_1,\ldots,a_j\in [n-1]$, $a_{j+1},\ldots,a_k\in [n]$, and $c_{j+1},\ldots,c_k\in \mathbb{Z}_r$, and define $Z=X_{a_1}\cdots X_{a_j}Y_{a_{j+1},c_{j+1}} \cdots Y_{a_k,c_k}$. Then on any conjugacy class $C_{\rpart}$ of $\mathfrak{S}_{n,r}$ without cycles of lengths $1,2,\ldots,j+k$, 
     \[\E_{\mathfrak{S}_{n,r}}[Z]=\E_{C_{\rpart}}[Z].\]
\end{lemma}

\begin{proof}
    We can assume without loss of generality that all of the elements $a_1,\ldots,a_j$ are distinct, and all of the elements $a_{j+1},\ldots,a_k$ are distinct. Starting with $Y$, observe that if some $a_i\in \{a_{j+1},\ldots,a_k\}$ is not in the essential set of $X_{a_1}\cdots X_{a_j}$, then \cref{lem:fmaj_reduction1} implies that it suffices to remove $Y_{a_i,c_i}$ and prove the result for the resulting statistic. Applying this repeatedly, we see that we can assume $\Ess(Y_{a_{j+1},c_{j+1}} \cdots Y_{a_k,c_k})\subseteq \Ess(X_{a_1}\cdots X_{a_j})$.

    Now suppose that this is a proper subset, so there exists some \[i\in \Ess(X_{a_1}\cdots X_{a_j})\setminus \{a_{j+1},\ldots,a_k\}.\] In this case, we can express
    \[Z=\sum_{c=0}^{r-1} Z\cdot Y_{i,c},\]
    where each statistic in the sum has degree at most $j+k$ by \cref{lem:XYproducts}. Hence, it suffices to show the statements for each $Z\cdot Y_{i,c}$. Iterating this process, we see that it suffices to consider when $\Ess(Y_{a_{j+1},c_{j+1}} \cdots Y_{a_k,c_k})=\Ess(X_{a_1}\cdots X_{a_j})$, and this case follows from \cref{lem:fmaj_reduction2}.
\end{proof}

\begin{theorem}\label{thm:fmaj}
    Let $C_{\rpart}$ be a conjugacy class of $\mathfrak{S}_{n,r}$ without cycles of lengths $1,2,\ldots,2k$. Then \[\E_{\mathfrak{S}_{n,r}}[\fmaj_{n,r}^k]=\E_{C_{\rpart}}[\fmaj_{n,r}^k].\]
\end{theorem}

\begin{proof}
    As noted in \eqref{eq:product}, $\fmaj_{n,r}^k$ can be expressed as linear combinations of the form \[X_{a_1}\cdots X_{a_j} Y_{a_{j+1},c_{j+1}}\cdots Y_{a_k,c_k}\]
    where $j\leq k$, $a_1,\ldots,a_j\in [n-1]$, $a_{j+1},\ldots,a_k\in [n]$, and $c_{j+1},\ldots,c_k\in \mathbb{Z}_r$. \cref{lem:final_fmaj} implies that on any $C_{\rpart}\subseteq \mathfrak{S}_{n,r}$ without cycles of lengths $1,2,\ldots,j+k$, 
    \begin{equation*}
        \E_{\mathfrak{S}_{n,r}}[X_{a_1}\cdots X_{a_j} Y_{a_{j+1},c_{j+1}}\cdots Y_{a_k,c_k}]=\E_{C_{\rpart}}[X_{a_1}\cdots X_{a_j}Y_{a_{j+1},c_{j+1}}\cdots Y_{a_k,c_k}].
    \end{equation*} 
    Since $j\leq k$, this holds on all $C_{\rpart}$ without cycles of lengths $1,2,\ldots,2k$. By linearity of expectation, we conclude $\E_{C_{\rpart}}[\fmaj_{n,r}^k]=\E_{\mathfrak{S}_{n,r}}[\fmaj_{n,r}^k]$.
\end{proof}

\begin{corollary}\label{cor:fmaj}
    For every $n\geq 1$, let $C_{\rpart_n}$ be a conjugacy class of $\mathfrak{S}_{n,r}$. Suppose that for all $i$, the number of cycles of length $i$ in $\rpart_n$ approaches 0 as $n\to\infty$. Then for sufficiently large $n$, $\fmaj_{n,r}$ has mean $\mu_{n,r}=\frac{n(rn+r-2)}{4}$ and variance $\sigma_{n,r}^2=\frac{2r^2n^3+3r^2n^2+(r^2-6)n}{72}$ on $C_{\rpart_n}$. Furthermore, as $n\to\infty$, the statistic $\frac{\fmaj_{n,r}-\mu_{n,r}}{\sigma_{n,r}}$  converges in distribution to the standard normal distribution. 
\end{corollary}

\begin{remark}\label{remark:total_order_choice}
    The original definitions of major index and flag major index given in \cite{AdinRoichman1} are based on the total order 
    \begin{equation}\label{eq:other_order_fmaj}
        1^{r-1} < 2^{r-1} <3^{r-1} <\cdots < 1^1 < 2^1 < 3^ 1< \cdots < 1^0<2^0<3^0\cdots .
    \end{equation}
    However, similar to many statistics on $\mathfrak{S}_n$, modifying the total ordering used in $\maj_{n,r}$ and $\fmaj_{n,r}$ does not affect the resulting distributions on $\mathfrak{S}_{n,r}$. 
    
    One method for showing this is to fix colors $c_1,\ldots,c_n\in \mathbb{Z}_r$ and partition $\mathfrak{S}_{n,r}$ into subsets of the form 
    \[\Omega_{(c_1,\ldots,c_n)} = \{(\omega,\tau)\in \mathfrak{S}_{n,r}: \{\omega(i)^{\tau(i)}\}_{i=1}^n =\{i^{c_i}\}_{i=1}^n\}.\]
    Any total order on $[n]^r$ will restrict to a total order on $\{i^{c_i}\}_{i=1}^n $. By replacing elements in the one-line notation of $(\omega,\tau)\in \Omega_{(c_1,\ldots,c_n)}$ with their images under the unique order-preserving map from $\{i^{c_i}\}_{i=1}^n$ to $[n]$, one can show that
    \begin{equation}\label{eq:maj_order}
        \sum_{(\omega,\tau)\in \Omega_{(c_1,\ldots,c_n)}} q^{\maj_{n,r}(\omega,\tau)} = \sum_{\omega\in \mathfrak{S}_n} q^{\maj_n(\omega)} = [1]_q[2]_q\cdots [n]_q,
    \end{equation}
    where $[i]_q=1+q+q^2+\cdots+q^{i-1}$ is the $q$-integer of $i$. For example, for the colored permutation $[3^1,8^0,5^0,6^1,2^2,1^2,4^0,7^1]\in \mathfrak{S}_{8,3}$ from \cref{ex:perm_example_1b}, this replacement results in the permutation $[4,3,2,5,8,7,1,6]\in \mathfrak{S}_8$. This new permutation has the same descent set as the original colored permutation, and hence, has the same major index statistic. 
    
    The corresponding result for $\fmaj_{n,r}$ with respect to any total order on $[n]^r$ is
    \begin{equation}\label{eq:fmaj_order}
        \begin{split}
            \sum_{(\omega,\tau)\in \Omega_{(c_1,\ldots,c_n)}} q^{\fmaj_{n,r}(\omega,\tau)} & = q^{c_1+c_2+\cdots+c_n} \sum_{(\omega,\tau)\in \Omega_{(c_1,\ldots,c_n)}} q^{r\cdot \maj_{n,r}(\omega,\tau)} \\
            & =  q^{c_1+c_2+\cdots+c_n} \prod_{i=1}^n (1+q^{r}+q^{2r}+\cdots+q^{(i-1)r}).
        \end{split}
    \end{equation}
    Since \eqref{eq:maj_order} and \eqref{eq:fmaj_order} hold regardless of the total order on $[n]^r$, it follows that the distributions of $\maj_{n,r}$ and $\fmaj_{n,r}$ coincide on any $\Omega_{(c_1,\ldots,c_n)}$ regardless of the total order. These distributions must therefore coincide on  $\mathfrak{S}_{n,r}$, so \cref{thm:fmaj_normal_Snr} holds regardless of the total order chosen for defining $\fmaj_{n,r}$.

    \cref{thm:fmaj} and \cref{cor:fmaj} can also be established for the flag major index statistic when defined using the total order in \eqref{eq:other_order_fmaj}. One can replace the usage of the the total order \eqref{eq:descentordering} with \eqref{eq:other_order_fmaj} in \algname{ColoredDescents} and prove analogs of the necessary results used throughout \cref{sec:descents,sec:flag-major}. More generally, our arguments can be adapted to establish \cref{thm:fmaj} and \cref{cor:fmaj} when the flag major index is defined with any total order that can be obtained from permuting the colors in \eqref{eq:descentordering}, changing the total order on $[n]$ used within all colors, or a combination of these two.
\end{remark}

\section{Conclusion}\label{sec:conclusion}

In this paper, we analyzed the moments and asymptotic distributions of $\des_{n,r}$ and $\fmaj_{n,r}$ on conjugacy classes $C_{\rpart}$ of $\mathfrak{S}_{n,r}$ with sufficiently long cycles. Our methods showed that the moments and asymptotic distributions of these statistics on $C_{\rpart}$ coincide with those on $\mathfrak{S}_{n,r}$. However, another natural question is to determine the actual distributions for these statistics on $C_{\rpart}$.

\begin{problem}\label{prob:1}
    Study the distributions of $\des_{n,r}$ and $\fmaj_{n,r}$ on conjugacy classes of $\mathfrak{S}_{n,r}$. 
\end{problem}

The distribution for $\des_{n}$ on conjugacy classes of $\mathfrak{S}_n$ was established by Diaconis, McGrath, and Pitman \cite{DiaconisPersiGrath}. Additionally, the distribution of $\des_{B_n}$ on conjugacy classes of $B_n$ was established by Campion Loth, Levet, Liu, Sundaram, and Yin \cite{GRWCColoredPermutationStatistics}, and this built on prior work of Reiner involving a different notion of descents on $B_n$ \cite{ReinerEuropJC1993-2}. As noted in the introduction, $\des_{B_n}$ does not coincide with $\des_{n,2}$, but the general approach may still be insightful for \cref{prob:1}.

Using the distribution of $\des_{n}$ on conjugacy classes of $\mathfrak{S}_n$, Kim and Lee \cite{KimLee2020} established asymptotic normality of the descent statistic on arbitrary conjugacy classes of $\mathfrak{S}_n$. Hence, one can consider the corresponding problem for $\des_{n,r}$ on arbitrary conjugacy classes of $\mathfrak{S}_{n,r}$, and results from \cref{prob:1} may be useful for this. 

\begin{problem}\label{prob:2}
    Determine the asymptotic distribution for $\des_{n,r}$ on arbitrary conjugacy classes of $\mathfrak{S}_{n,r}$.
\end{problem}

\section*{Acknowledgements}

We would like to thank Sara Billey, Michael Levet, Ricky Liu, Shayan Oveis Gharan, Farbod Shokrieh, and Sheila Sundaram for helpful feedback and discussions. Mei Yin's research was partially supported by Simons Foundation Grant MPS-TSM-00007227. An extended abstract of this paper appears in \emph{S\'{e}minaire Lotharingien de Combinatoire} as part of the Proceedings for the 37th International Conference on Formal Power Series and Algebraic Combinatorics (FPSAC) \cite{desfmaj_fpsac}.

\bibliographystyle{plain}
\bibliography{Bibliography}

\end{document}